\pdfoutput=1
\documentclass[11pt]{amsart}
\usepackage[backend=biber,
	    url=false,
	    doi=false,
	    isbn=false,
	    date=year,
	    maxbibnames=99,
	    giveninits,
	    sortcites=true,
	    sorting=nty,
	    style=numeric]{biblatex}

\renewbibmacro{in:}{}
            
\usepackage[utf8]{inputenc}
\usepackage[T1]{fontenc}
\usepackage[text={445pt,630pt}, centering, a4paper, marginparwidth=2cm]{geometry}
\usepackage[hidelinks,pdfusetitle]{hyperref}
\usepackage[activate={true,nocompatibility},final,tracking=true,kerning=true,factor=1100,stretch=10,shrink=10]{microtype}
\usepackage{braket}    
\usepackage{mathtools} 
\usepackage{amsthm}
\usepackage{amssymb}
\usepackage{mathrsfs}  
\usepackage{dsfont}    
\usepackage{bbold}

\usepackage[shortlabels]{enumitem}
\usepackage[justification=centering]{caption}
\usepackage[margin=5pt,justification=centering,labelformat=simple,labelfont=sc]{subcaption}

\theoremstyle{plain}
\newtheorem{theorem}{Theorem}[section]
\newtheorem{corollary}[theorem]{Corollary}

\newtheorem{proposition}[theorem]{Proposition}
\newtheorem{lemma}[theorem]{Lemma}
\newtheorem{lem}[theorem]{Lemma}

\newtheorem{definition}[theorem]{Definition}

\theoremstyle{remark}
\newtheorem{remark}[theorem]{Remark}
\newtheorem{rem}[theorem]{Remark}
\newtheorem{example}[theorem]{Example}

\newlist{propenum}{enumerate}{1} 
\setlist[propenum]{label=(\roman*)}

\DeclareMathOperator\diag{Diag}               
\DeclareMathOperator\spec{Spec}               
\DeclareMathOperator\mult{m}               
\DeclarePairedDelimiter\abs{\lvert}{\rvert}   
\DeclarePairedDelimiter\norm{\lVert}{\rVert}  

\DeclarePairedDelimiter\floor{\lfloor}{\rfloor}

\newcommand{\cb}{\ensuremath{\mathscr{B}}}

\newcommand{\cf}{\ensuremath{\mathscr{F}}}

\newcommand{\ck}{\ensuremath{\mathscr{K}}}
\newcommand{\ind}[1]{\mathds{1}_{#1}}
\newcommand{\ca}{\ensuremath{\mathscr{A}}}

\newcommand{\un}{\ensuremath{\mathds{1}}}

\newcommand{\cpa}{\mathcal{P}^\mathrm{Anti}}
\newcommand{\F}{\ensuremath{\mathcal{F}}}
\newcommand{\AF}{\ensuremath{\mathcal{F}^\mathrm{Anti}}}
\newcommand{\kk}{\ensuremath{\mathrm{k}}}
\newcommand{\rd}{\ensuremath{\mathrm{d}}}
\newcommand{\cll}{\ensuremath{\mathcal{L}}}

\newcommand{\R}{\ensuremath{\mathbb{R}}}

\newcommand{\C}{\ensuremath{\mathbb{C}}}
\newcommand{\N}{\ensuremath{\mathbb{N}}}

\newcommand{\cp}{\ensuremath{\mathcal{P}}}

\newcommand{\FF}{\ensuremath{\mathbf{F}}}
\newcommand{\as}{\text{ a.s.}}
\newcommand{\lb}{[\![}

\newcommand{\costu}{\ensuremath{C_\mathrm{uni}}} 
\newcommand{\zero}{\ensuremath{\mathbb{0}}}
\newcommand{\cmir}{\ensuremath{c_\star}}
\newcommand{\cmar}{\ensuremath{c^\star}}
\newcommand{\cmax}{c_{\max}}
\newcommand{\mir}{\ensuremath{R_{e\star}}}
\newcommand{\mar}{\ensuremath{R_e^\star}}
\newcommand{\Cinf}{\ensuremath{C_\star}} 
\newcommand{\Csup}{\ensuremath{C^\star}} 
\newcommand{\oa}{\ensuremath{\Omega_\mathrm{a}}}
\newcommand{\oi}{\ensuremath{\Omega_\mathrm{i}}}

\newcommand{\pos}{\ensuremath{\mathrm{p}}}
\newcommand{\nega}{\ensuremath{\mathrm{n}}}

\date{\today}
\author{Jean-François Delmas}
\address{Jean-François Delmas,
  CERMICS, \'{E}cole des Ponts, France}
\email{jean-francois.delmas@enpc.fr}

\author{Dylan Dronnier}
\address{Dylan Dronnier,
  CERMICS, \'{E}cole des Ponts, France}
\email{dylan.dronnier@enpc.fr}

\author{Pierre-André Zitt}
\address{Pierre-André Zitt, LAMA, Université Gustave Eiffel, France}
\email{pierre-andre.zitt@univ-eiffel.fr}

\newcommand{\loss}{{R_e}}
\newcommand{\grR}{R_e[\kk]}
\newcommand{\grS}{\spec[\kk]}

\newsavebox{\largestimage}

\title[Effective reproduction number]{Effective reproduction number: convexity, invariance and
  cordons sanitaires} 

\begin{document}

\thanks{This work is partially supported by Labex Bézout reference ANR-10-LABX-58}



\begin{abstract} 
  We consider the problem of optimal allocation strategies for a (perfect) vaccine in an
  infinite-metapopulation model (including SIS, SIR, SEIR, \ldots), when the loss function is
  given by the effective reproduction number $R_e$, which is defined as the spectral
  radius of the effective next generation matrix (in finite dimension) or more generally
  of the effective next generation operator (in infinite dimension). We give sufficient
  conditions for~$R_e$ to be a convex or a concave function of the vaccination strategy.
  Then, following a previous work, we consider the bi-objective problem of minimizing
  simultaneously the cost and the loss of the vaccination strategies. In particular, we
  prove that a cordon sanitaire might not be optimal, but it is still better than the 
  ``worst'' vaccination strategies. Inspired by the graph theory, we compute the minimal
  cost which ensures that no infection occurs using independent sets. Using Frobenius
  decomposition of the whole population into irreducible sub-populations, we give some
  explicit formulae for optimal (``best'' and ``worst'') vaccinations strategies.
  Eventually, we provide equivalence properties on models which ensure that the function
  $R_e$ is unchanged.
\end{abstract}

\maketitle

\section{Introduction}

\subsection{Vaccination in metapopulation models}

The study of vaccination strategies for metapopulation models with $N\geq 2$
sub-populations, naturally leads to an easily stated linear algebra problem: given a 
matrix $K$, of size $N\times N$, with non-negative entries, what can be said about the
function
\begin{equation}\label{eq:informalPb}
  R_e:
  \begin{cases}
    \Delta &\to \mathbb{R}, \\
    \eta &\mapsto \text{spectral radius of } K\cdot \diag(\eta),
  \end{cases}
\end{equation}
where $\Delta= [0,1]^N $, $\diag(\eta)$ denotes the $N\times N$ matrix with diagonal
elements $\eta=(\eta_1, \ldots, \eta_N)$, and the spectral radius is the largest modulus
of the eigenvalues. In this form, the problem appears for instance, with a mathematical
point of view, in Elsner and Hadeler~\cite{elsner_15}, see also Friedland~\cite{Fried80}
and Nussbaum~\cite{nussbaum_cvx}. 

\medskip

In metapopulation epidemiological models, the indices $i=1,\ldots, N$ correspond to
various sub-populations with respective proportional size $\mu_1, \ldots, \mu_{N}$. 
Following \cite{hill-longini-2003}, the entry $K_{ij}$ of the so-called next-generation
matrix $K$ is equal to the expected number of secondary infections for people in subgroup
$i$ resulting from a single randomly selected non-vaccinated infectious person in subgroup
$j$. Finally, $\eta$ represents a vaccination strategy, that is, $\eta_i$ is the
\textbf{fraction of non-vaccinated} individuals in the $i$\textsuperscript{th}
sub-population; \textbf{thus $\eta_i = 0$ when the $i$\textsuperscript{th} sub-population
is fully vaccinated, and $1$ when it is not vaccinated at all}. (This seemingly unnatural
convention is in particular motivated by the simple form of
Equation~\eqref{eq:informalPb}). So, the strategy $\un\in \Delta$, with all its entries
equal to 1, corresponds to an entirely non-vaccinated population. The quantity $R_e$,
referred to as the \emph{effective reproduction number}, may then be interpreted as the
mean number of infections coming from a typical case. In particular, we denote by
$R_0=R_e(\un)$ the so-called \emph{basic reproduction number} associated to the
metapopulation epidemiological model. With the interpretation of the function~$R_e$ in
mind, it is then very natural to minimize it under a constraint on the cost $C(\eta)$ of
the vaccination strategies $\eta$. A natural choice for the cost function is given by the
uniform cost $C(\eta)=1- \sum_i \eta_i \mu_i$, which corresponds to the fraction of
vaccinated individuals in the population. This constrained optimization problem appears in
most of the literature for designing efficient vaccination strategies for multiple
epidemic situation (SIR/SEIR), see \cite{EpidemicsInHeCairns1989, hill-longini-2003,
TheMostEfficiDuijze2016, CriticalImmuneMatraj2012, poghotanyan_constrained_2018,
OptimalInfluenEnayat2020, IdentifyingOptZhao2019}. Note that in some of these references,
the effective reproduction number is defined as the spectral radius of the matrix
$\diag(\eta) \cdot K$. Since the eigenvalues of $\diag(\eta) \cdot K$ are exactly the
eigenvalues of the matrix $K\cdot \diag(\eta)$, this actually defines the same function
$R_e$. In Section~\ref{sec:ST-TNG}, we discuss the generalization of the effective
reproduction number to the kernel model that offers a finer description of the contacts
within the population.

\medskip

The goal of this paper is to prove a number of properties of~$R_e$, that shed a light on
how to vaccinate in the best possible way. In previous
works~\cite{delmas_infinite-dimensional_2020, ddz-theo}, we introduced a general
infinite-dimensional kernel framework in which the matrix formulation appears as a special
finite-dimensional case. We state our results in this general framework, but for ease of
the presentation, we shall stick to the matrix formulation in this introduction. Finally,
the results of this paper are applied and illustrated in detail on various examples in the
companion papers~\cite{ddz-reg, ddz-mono, ddz-2pop}.

\subsection{Convexity properties of the effective reproduction number}

Given the importance of convexity to solve optimization problems efficiently, it is
natural to look for conditions on the matrix $K$ that imply convexity or concavity for the
map~$R_e$ defined by \eqref{eq:informalPb}. In their investigation of the behavior of this
map in the finite dimensional matrix setting, Hill and Longini conjecture
in~\cite{hill-longini-2003} sufficient spectral conditions to get either concavity or
convexity. More precisely, guided by explicit examples, they state that $R_e$ should be
convex if all the eigenvalues of $K$ are non negative real numbers, and that it should be
concave if all eigenvalues are real, with only one positive eigenvalue.

Our first series of results show that, while this conjecture cannot hold in full
generality, see Section~\ref{sec:comparison_Hill_Longini}, it is true under an additional
symmetry hypothesis. Recall that a matrix $K$ is called diagonally symmetrizable if there
exist positive numbers $(d_1,\ldots d_{N})$ such that for all $i,j$, $d_i K_{ij} = d_j
K_{ji}$. Such a matrix is necessarily diagonalizable with real eigenvalues. The following
result, which appears below in the text as Theorem~\ref{th:hill-longini-meta}, settles the
conjecture for diagonally symmetrizable matrices. It is a special case of the more general
Theorem~\ref{th:hill-longini}, which holds in the infinite dimensional kernel setting, and
for which the symmetry assumption has to be carefully worded. Let us mention that the
eigenvalue $\lambda_1$ in the theorem below is non-negative and is equal to the spectral
radius of $K$, that is, $\lambda_1=R_e(\un)=R_0$, thanks to the Perron-Frobenius theory. 

\begin{theorem}
  Let $K$ be an $N\times N$ matrix with non-negative entries. Suppose that $K$ is
  diagonally symmetrizable with eigenvalues $\lambda_1\geq\lambda_2\cdots \geq
  \lambda_{N}$.
  \begin{propenum}
  \item\label{item:cairns} If $\lambda_{N}\geq 0$, then the function $R_e$ is convex. 
  \item If $\lambda_2 \leq 0$, then the function $R_e$ is concave. 
  \end{propenum}
\end{theorem}

Note that the case~\ref{item:cairns} appears already in
Cairns~\cite{EpidemicsInHeCairns1989}; see also \cite{Fried80,feng_elaboration_2015} and 
Section~\ref{sec:comparison_Hill_Longini} below for a detailed comparison with existing
results.

\medskip

It is easy to see that if $K$ and $K'$ are diagonally similar up to transposition, they
define the same function~$R_e$ (see \cite{ddz-pres} for more results in this direction).
We check in Section \ref{sec:equivalent} that this is essentially still true in the
generalized kernel setting.

\subsection{Properties of Pareto and anti-Pareto optima, cordons sanitaires}

Let us now come back to the problem of finding optimal vaccination strategies. In contrast
with our previous work~\cite{ddz-theo}, where we put minimal assumptions on the loss
function which measures the efficiency of the vaccination strategies, we consider here
that the loss of a strategy $\eta$ is given by its effective reproduction number
$R_e(\eta)$. This focus and the fact that we consider strictly decreasing cost functions
(because vaccinating more costs more, see Section~\ref{sec:P-et-AP}), allow us to simplify
some of the statements of~\cite{ddz-theo} and to give additional specific results.

The problem of minimizing the effective reproduction number while keeping the cost of the
vaccination low leads to a bi-objective optimization problem. We recall in
Section~\ref{sec:P-et-AP} the setting introduced in detail in~\cite{ddz-theo} for a
general framework. One can identify Pareto optimal and anti-Pareto optimal vaccinations
strategies, informally ``best'' and ``worst'' vaccination strategies, and consider the
Pareto frontier $\F$ (resp. anti-Pareto frontier $\AF$) as the outcomes $(C(\eta),
R_e(\eta))$ of the Pareto (resp. anti-Pareto) optimal strategies $\eta$.

In Figure~\ref{fig:sym-circle-pareto}, we have plotted in red the Pareto frontier and in a
dashed red line the anti-Pareto frontier when the next-generation matrix is the adjacency
matrix of the non-oriented cycle graph with $N=12$ nodes from Figure~\ref{fig:cycle-graph}
and Example~\ref{ex:cycle-graph}, see also Example~\ref{exple:meta}.

\begin{figure}
  \begin{subfigure}[T]{.5\textwidth}
    \centering
    \includegraphics[page=5]{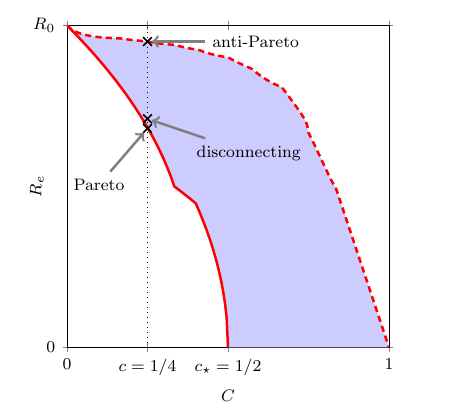}
    \caption{Thick red line: Pareto frontier; Dashed line: anti-Pareto
      frontier; x marker: 
    outcomes of various strategies; light blue: all possible outcomes
    $(C(\eta), R_e(\eta))$ for $\eta\in \Delta$.}
    \label{fig:sym-circle-pareto}
  \end{subfigure}%
  \begin{subfigure}[T]{.5\textwidth}
    \centering
    \includegraphics[page=6]{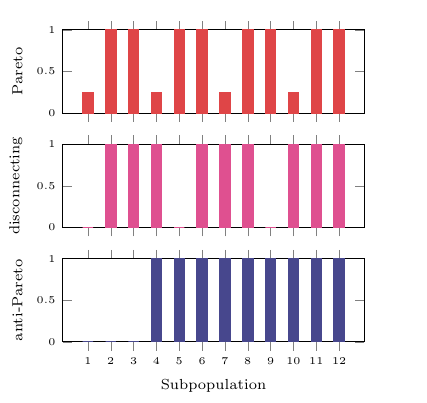}
    \caption{Profile of various strategies.}
    \label{fig:sym-circle-strategies}
  \end{subfigure}
  \caption{Performance of the disconnecting vaccination strategy ``one
    in $4$'' for the non-oriented 
    cycle graph with 12 nodes and uniform cost $1/4$.}
  \label{fig:perf}
\end{figure}

\subsubsection{A cordon sanitaire is not the worst vaccination strategy}

Recall that a matrix $K$ is reducible if there exists a permutation $\sigma$ such that
$(K_{\sigma(i)\sigma(j)})_{i,j}$ is block upper triangular, and irreducible otherwise. A
\emph{cordon sanitaire} is a vaccination strategy $\eta$ such that the infection matrix
between non-vaccinated people, $K \cdot \diag(\eta)$, is reducible: informally, such a
vaccination cuts the effective population in two or more groups that do not infect one
another.

Disconnecting the population by creating a cordon sanitaire is not always the ``best''
choice, that is, it may not be Pareto optimal. However, we prove in
Proposition~\ref{prop:cut} that a cordon sanitaire can never be anti-Pareto optimal; this
result still holds in the general kernel framework, provided that the definition of cordon
sanitaires is generalized in an appropriate way.

\begin{example}[Non-oriented cycle graph]\label{ex:cycle-graph}
  Suppose that the matrix $K$ is given by the adjacency matrix (see
  Figure~\ref{fig:kernel-cycle} for a grayplot representation) of the non-oriented cycle
  graph with $N=12$ nodes; see Figure~\ref{fig:cycle-graph}. For a cost $\costu = 1/4$,
  there is a disconnecting strategy $\eta$ that consists in vaccinating one sub-population
  in four; see Figure~\ref{fig:cycle-disc} (and Figure~\ref{fig:kernel-sep} for a grayplot
  representation of the corresponding adjacency matrix). The effective reproduction number
  associated is equal to $\sqrt{2}$. This strategies performs better than the anti-Pareto
  optimal strategy and is out-performed by the Pareto optimal one as we can see in
  Figure~\ref{fig:perf}. This example is discussed in detail in
  \cite[Section~2.4]{ddz-reg}.
\end{example}

\begin{figure}
  \begin{subfigure}[T]{.5\textwidth}
    \centering
    \includegraphics[page=1]{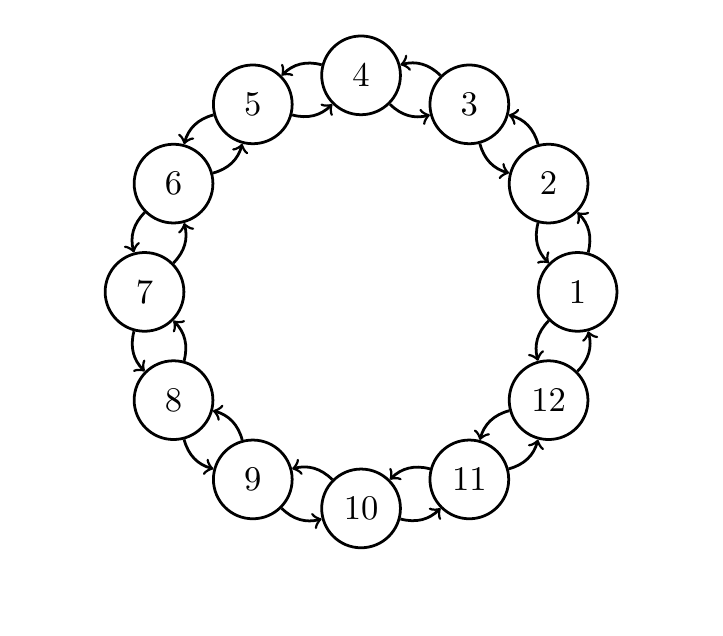}
    \caption{The non-oriented cycle graph.}\label{fig:cycle-graph} 
  \end{subfigure}%
  \begin{subfigure}[T]{.5\textwidth}
    \centering
    \includegraphics[page=3]{cordon}
    \caption{Grayplot of the corresponding kernel.}
    \label{fig:kernel-cycle}
  \end{subfigure}

  \begin{subfigure}[T]{.5\textwidth}
    \centering
    \includegraphics[page=2]{cordon}
    \caption{Cordon sanitaire corresponding to the
      ``one in $4$'' vaccination strategy (in green the vaccinated groups).}
    \label{fig:cycle-disc} 
  \end{subfigure}%
  \begin{subfigure}[T]{.5\textwidth}
    \centering
    \includegraphics[page=4]{cordon}
    \caption{Grayplot of the corresponding kernel.}
    \label{fig:kernel-sep}
  \end{subfigure}
  \caption{Example of disconnecting vaccination strategy on the
    non-oriented cycle graph with $N=12$ nodes.}
  \label{fig:cycle-kern-graph-disc}
\end{figure}

\subsubsection{Minimal cost required
to completely stop the transmission of the disease}

A vaccination strategy $\eta$ such that $R_e(\eta)=0$ completely eradicates the epidemic.
Section~\ref{sec:independance} is devoted to the characterization of the minimal cost of
such vaccinations, which is denoted by~$\cmir$. This quantity is introduced and discussed
in~\cite{ddz-theo} under general assumption for the loss function. Since we consider here
the special case of measuring the loss by the effective reproduction number~$R_e$, we are
able to give in Proposition~\ref{prop:CRe(0)} an explicit expression of this quantity in
the kernel model. In the symmetric matrix case, when the cost is uniform (the cost is
proportional to the number of vaccinated individuals), this expression is proportional to
the size of maximal independent sets of the non-oriented graph with vertices $\{1,\ldots,
N\}$, where there is an edge between $i$ and $j$ if and only if $K_{ij}>0$.

We can observe this property in Figure~\ref{fig:sym-circle-pareto} as the size of the
maximal independent set of the non-oriented cycle graph of size $N$ from
Example~\ref{ex:cycle-graph} is equal to $\floor{N/2}$.

\subsubsection{Reducible case}

When the matrix $K$ happens to be reducible, up to a relabeling, we may assume that it is
block upper triangular. Denoting by $m$ the number of blocks and $I_1, \ldots, I_m$ the
sets of indices describing the blocks, this means that for all $\ell>k$ and $(i,j)\in
I_\ell\times I_k$, we have $K_{ij} = 0$. In the epidemiological interpretation, this means
that the populations with indices in $I_k$ never infect the ones with indices in $I_\ell$.
One may then hope that the study of $R_e$ can be effectively reduced to the study of the
effective radius of the square sub-matrices $(K_{ij})_{i,j\in I_k}$ describing the
infections within block $I_k$. This is indeed the case, and we give in
Section~\ref{sec:reducible} a complete picture of the Pareto and anti-Pareto frontiers of
$R_e$, in terms of the effective reproduction numbers restricted to each irreducible
component of the infection kernel or matrix. In particular, this allows a better
understanding of the possible disconnection of the anti-Pareto frontier, whereas the
Pareto frontier is always connected. Once more, special care has to be taken with the
definitions when handling the infinite dimensional kernel case.

\subsubsection{Optimal ray}

It is observed by Poghotanyan, Feng, Glasser and Hill in
\cite{poghotanyan_constrained_2018}, that if there exists a Pareto optimal strategy $\eta$
with all its entries strictly less than 1, then all the strategies $\lambda \eta$, with
$\lambda\geq 0$ such that $\lambda \eta\in \Delta$, are Pareto optimal. We give a short
proof on the existence of such optimal rays in Section~\ref{sec:ray}, when one assumes
that the cost function $C$ is affine on $\Delta$.

\subsection{Structure of the paper}

We discuss in Section~\ref{sec:ST-TNG} the generality of the setting, showing that
studying vaccination strategies in many different epidemic models gives rise to the same
optimization problem. After recalling formally our infinite dimensional kernel setting in
Section~\ref{sec:settings}, we discuss invariance properties of $R_e$ in
Section~\ref{sec:equivalent}. The convexity properties of $R_e$ and the related conjecture
of Hill and Longini are discussed in Section~\ref{sec:Hill_Longini}. Various properties of
the Pareto and anti-Pareto frontiers, and in particular the fact that establishing a
\emph{cordon sanitaire} by disconnecting the population is never the worst solution, are
discussed in Section~\ref{sec:3-P-et-AP}. Finally, the case of reducible kernels is
treated in Section~\ref{sec:reducible}.

\section{Discussion on the next-generation operator}\label{sec:ST-TNG}

In \cite{delmas_infinite-dimensional_2020, ddz-theo}, we developed a framework that we
call the kernel model where the population is represented as an abstract probability space
$(\Omega, \mathscr{F}, \mu)$. Individuals are characterized by a feature $x \in \Omega$,
and the relative size of the sub-population with feature $x$ is given by $\mu(\mathrm{d}
x)$. The underlying structure described by this feature can be very varied, typical
examples being one or several of the following characteristics: spatial position, social
contacts, susceptibility, infectiousness, characteristics of the immunological response,
\ldots The analogue of the next-generation matrix $K$ is then the kernel operator defined
formally by:
\[
  T_\kk ( g)(x) = \int_\Omega \kk(x,y) \, g(y) \,\mathrm{d}\mu(y);
\]
where the non-negative kernel $\kk$ is defined on $\Omega\times \Omega$ and $\kk(x,y)$
still represents a strength of infection from $y$ to $x$. Vaccination strategies $\eta \,
\colon \, \Omega \to [0,1]$ encode the density of non-vaccinated individuals with respect
to the measure~$\mu$. The (sub-probability) measure $\eta(y)\, \mu(\mathrm{d} y)$ may then
be understood as an effective population, giving rise to an effective next-generation
operator:
\[
  T_{\kk\eta} ( g)(x) = \int_\Omega \kk(x,y) \, g(y) \, \eta(y) \,
  \mu(\mathrm{d} y).
\]
The effective reproduction number is then defined by $R_e(\eta) = \rho(T_{\kk \eta})$,
where $\rho$ stands for the spectral radius of the operator and $\kk \eta$ for the kernel
$(\kk \eta)(x,y)=\kk(x,y) \eta(y)$.

\medskip

Most of the results mentioned in the introduction will be given in this general framework
as we argue that the latter is sufficiently flexible to describe a wide range of epidemic
models from the literature including the metapopulation models. We give in the following a
few examples to support this claim: in each of them, the spectral radius of a particular,
explicit kernel operator appears as a threshold parameter, and the epidemic either
``invades/survives'' or ``dies out'' depending on the value of this parameter. Classical
notations are used: $S$ denotes the proportion of susceptible individuals, $E$ the
proportion of those who have been exposed to the disease, $I$ the proportion of infected
individuals, $R$ the proportion of removed individuals in the population.

\begin{example}[Meta-population models]\label{exple:meta}
  Recall that in metapopulation models, the population is divided into $N \geq 2$
  different sub-populations of respective proportional size $\mu_1, \ldots ,\mu_N$, and
  the reproduction number is given by \( R_e(\eta) = \rho(K \cdot\diag(\eta)) \), where
  $K$ is the next generation matrix and $\eta$ belongs to $[0,1]^N$ and gives the
  proportion of non-vaccinated individuals in each sub-population. To express the function
  $R_e$ as the effective reproduction number of a kernel model, consider the discrete
  state space $\Omega_{\mathrm{d}} = \{ 1, \ldots, N \}$ equipped with the probability
  measure $\mu_{\mathrm{d}}$ defined by $\mu_{\mathrm{d}}(\{i\}) = \mu_i$, and let
  $\kk_{\mathrm{d}}$ denote the discrete kernel on $\Omega_{\mathrm{d}}$ defined by:
  \begin{equation}\label{eq:next-kernel} \kk_{\mathrm{d}}(i,j) = K_{ij}/\mu_j.
  \end{equation}
  For all $\eta \in \Delta = [0,1]^N$, the matrix $K\cdot \diag(\eta)$ is the matrix
  representation of the endomorphism $T_{\kk_{\mathrm{d}} \eta}$ in the canonical basis of
  $\R^N$. In particular, we have: $R_e(\eta) = \rho(T_{\kk \eta}) = \rho(K
  \cdot\diag(\eta))$.

  In Figure~\ref{fig:kernel-cycle}, we have plotted the kernel on $[0, 1]$ associated to
  $\kk_{\mathrm{d}}$ for the non-oriented cycle graph when the sub-populations have the
  same size.
\end{example}

\begin{example}[An SIR model with nonlinear incidence rate and vital dynamics]
  In \cite{thieme_global_2011}, Thieme proposed an SIR model in an infinite-dimensional
  population structure with a nonlinear incidence rate. The structure space is given by
  $\Omega$ a compact subset of $\R^N$ equipped with the normalized Lebesgue measure. We
  restrict slightly his assumption so that the incidence rate is a linear function of the
  number of susceptible. The dynamic of the epidemic then writes:
  \begin{equation}
    \text{For $ t \geq 0$, $x \in \Omega$,}\qquad \left\{
      \begin{array}{ll}
	\partial_t S(t,x) = \Lambda(x) - \nu(x) S(t,x) - S(t,x)\, \int_\Omega f(I(t,y), x,
	y) \, \mathrm{d}y, \\
	\\
	\partial_t I(t,x) = S(t,x)\, \int_\Omega f(I(t,y), x, y) \, \mathrm{d}y - (\gamma(x)
	+ \nu(x))I(t,x), \\
	\\
	\partial_t R(t,x) = \gamma(x) I(t,x).
      \end{array}
    \right.
  \end{equation}
  Here $\Lambda(x)$ is the rate at which fresh susceptible individuals are recruited into the
  population at location $x$, $\nu(x)$ is the \emph{per capita} death rate of the
  individuals, and $\gamma(x)$ is the \emph{per capita} recovery rate of infectious
  individuals The integral term describes the incidence at $x$ at time $t$, \emph{i.e.},
  the rate of new infections. Thieme identified a threshold parameter that plays the role
  of the reproduction number, and is given by the spectral radius of the operator $T_\kk$
  with the kernel given by:
  \begin{equation}
    \kk(x,y) = \frac{\Lambda(x)}{\gamma(x) + \nu(x)} \partial_I f(0, x,y), \quad x,y \in
    \Omega,
  \end{equation}
  where $\partial_I f(0, x,y)$, the derivative of $f$ with respect to $I$, is supposed to
  be non-negative.

  Suppose that individuals at location $x$ are vaccinated with
  probability $1-\eta(x)$ at birth so that the susceptible individuals
  with feature $x$ are recruited at rate $\eta(x) \Lambda(x)$ and
  recovered/immunized individuals are also recruited at rate
  $(1-\eta(x)) \Lambda(x)$ at location $x$. The threshold parameter
  $R_e(\eta)$ is then given by the spectral radius of the integral
  operator $T_{\eta \kk}$ with kernel $\eta\kk $ given by
  $(\eta\kk)(x,y)=\eta(x) \kk(x,y)$. According to
  Lemma~\ref{lem:prop-spec-mult}~\ref{item:spec-adjoint-mult}, we have
  $\rho(T_{\eta \kk}) = \rho(T_{\kk \eta})$, and our framework can be
  used for this model.

  Under regularity assumptions on the parameters of the model, Thieme proved that if
  $R_e(\eta)$ is greater than $1$, then there exists an endemic equilibrium that attracts
  all the solutions while if $R_e(\eta)$ is smaller than $1$, then $I(t,x)$ converges to $0$
  for all $x \in \Omega$ as $t$ goes to infinity.
\end{example}

\begin{example}[An SEIR model without vital dynamics]
  In \cite{FinalSizeAndAlmeid2021}, Almeida, Bliman, Nadin and Perthame studied an
  heterogeneous SEIR model where the population is again structured with a bounded subset
  $\Omega \subset \R^N$ equipped with the normalized Lebesgue measure. The dynamic of the
  susceptible, exposed, infected and recovered individuals writes:
  \begin{equation}
    \text{For $t \geq 0$, $x \in \Omega$,}\qquad \left\{
      \begin{array}{ll}
	\partial_t S(t,x) = - S(t,x) \,\int_\Omega k(x,y) I(t,y)\, \mathrm{d}y, \\
	\\
	\partial_t E(t,x) = S(t,x) \, \int_\Omega k(x,y) I(t,y)\, \mathrm{d}y - \alpha(x) E(t,x), \\
	\\
	\partial_t I(t,x) = \alpha(x) E(t,x) - \gamma(x) I(t,x), \\
	\\
	\partial_t R(t,x) = \gamma(x) I(t,x).
      \end{array}
    \right.
  \end{equation}
  Here, the average incubation rate is denoted by $\alpha(x)$ and the average recovery
  rate by $\gamma(x)$; both quantities may depend upon the trait $x$. The function $k$
  is the transmission kernel of the disease. In this model, the basic reproduction number
  is given by the spectral radius of the integral operator $T_\kk$ with kernel
  $\kk=k/\gamma$:
  \begin{equation}\label{eq:SEIR} \kk(x,y) = k(x,y)/\gamma(y).
  \end{equation}
  Suppose that, prior to the beginning of the epidemic, the decision maker immunizes a
  density $1 - \eta$ of individuals. According to
  \cite[Section~3.2]{FinalSizeAndAlmeid2021}, the effective reproduction number is given
  by $\rho(T_{\eta \kk})$ which is also equal to $\rho(T_{\kk \eta})$, see
  Lemma~\ref{lem:prop-spec-mult}~\ref{item:spec-adjoint-mult} below, and our model is
  indeed suitable for studying the vaccination strategies in this context.
\end{example}

\begin{example}[An SIS model without vital dynamic]\label{exple:sis}
  In \cite{delmas_infinite-dimensional_2020}, generalizing the discrete model of
  Lajmanovich and Yorke~\cite{lajmanovich1976deterministic}, we introduced the following
  heterogeneous SIS model where the population is structured with an abstract probability
  space~$(\Omega, \mathscr{F}, \mu)$:
  \begin{equation}
    \text{For $t \geq 0$, $x \in \Omega$,}\qquad \left\{
      \begin{array}{ll}
	\partial_t S(t,x) = - S(t,x) \,\int_\Omega k(x,y) I(t,y)\, \mathrm{d}y + \gamma(x)
	I(t,x), \\
	\\
	\partial_t I(t,x) = S(t,x) \,\int_\Omega k(x,y) I(t,y)\, \mathrm{d}y - \gamma(x)
	I(t,x).
      \end{array}
    \right.
  \end{equation}
  The function $\gamma$ is the \emph{per-capita} recovery rate and $k$ is the transmission
  kernel. For this model, $R_e(\eta) = \rho(T_{\kk \eta})$ where $\kk=k/\gamma $ is
  defined by~\eqref{eq:SEIR}.

  Suppose that, prior to the beginning of the epidemic, a density $1 - \eta$ of
  individuals is vaccinated with a perfect vaccine. In the same way as for the SEIR model,
  we proved, as $t$ goes to infinity, that if $R_e(\eta)$ is smaller than or equal to $1$,
  then $I(t,\cdot)$ converges to $0$, and, under a connectivity assumption on the kernel
  $k$, that if $R_e(\eta)$ is greater than $1$, then $I(t,\cdot)$ converges to an endemic
  equilibrium. This highlights the importance of $R_e$ in the design of vaccination
  strategies.
\end{example}

\section{Setting, notations and previous results}\label{sec:settings}

\subsection{Spaces, operators, spectra}\label{sec:spaces}

All metric spaces~$(S,d)$ are endowed with their Borel~$\sigma$-field denoted by $\cb(S)$.
The set $\ck$ of compact subsets of~$\C$ endowed with the Hausdorff distance
$d_\mathrm{H}$ is a metric space, and the function~$\mathrm{rad}$ from~$\ck$ to $\R_+$
defined by~$\mathrm{rad}(K)=\max\{|\lambda|\, ,\, \lambda\in K\}$ is Lipschitz continuous
from~$(\ck,d_\mathrm{H})$ to~$\R$ endowed with its usual Euclidean distance.

Let~$(\Omega, \cf, \mu)$ be a probability space. We denote by~$\Delta$ the set of $[0,
1]$-valued measurable functions defined on~$\Omega$. For~$f$ and~$g$ real-valued functions
defined on~$\Omega$, we may write~$\langle f, g \rangle$ or $\int_\Omega f g \, \mathrm{d}
\mu$ for $\int_\Omega f(x) g(x) \,\mu( \mathrm{d} x)$ whenever the latter is meaningful.
For~$p \in [1, +\infty]$, we denote by $L^p=L^p( \mu)=L^p(\Omega, \mu)$ the space of
real-valued measurable functions~$g$ defined~$\Omega$ such that $\norm{g}_p=\left(\int
|g|^p \, \mathrm{d} \mu\right)^{1/p}$ (with the convention that~$\norm{g}_\infty$ is
the~$\mu$-essential supremum of $|g|$) is finite, where functions which agree~$\mu$-a.s.\
are identified. We denote by~$L^p_+$ the subset of~$L^p$ of non-negative functions.

\medskip

Let~$(E, \norm{\cdot})$ be a Banach space. We denote by~$\norm{\cdot}_E$ the operator norm
on~$\cll(E)$ the Banach algebra of bounded operators. The spectrum~$\spec(T)$ of~$T\in
\cll(E)$ is the set of~$\lambda\in \C$ such that~$ T - \lambda \mathrm{Id}$ does not have
a bounded inverse operator, where~$\mathrm{Id}$ is the identity operator on~$E$. Recall
that~$\spec(T)$ is a compact subset of~$\C$, and that the spectral radius of~$T$ is given
by:
\begin{equation}\label{eq:def-rho}
  \rho(T)=\mathrm{rad}(\spec(T))=
  \lim_{n\rightarrow \infty } \norm{T^n}_E^{1/n}.
\end{equation}
The element $\lambda\in \spec(T)$ is an eigenvalue if there exists $x\in E$ such that
$Tx=\lambda x$ and $x\neq 0$. Following \cite{konig}, we define the (algebraic)
multiplicity of $\lambda\in \C$ by:
\[
  \mult(\lambda, T)= \dim \left( \bigcup_{k\in \N^*} \ker (T- \lambda
  \mathrm{Id})^k\right),
\]
so that $\lambda$ is an eigenvalue if $\mult(\lambda, T)\geq 1$. We say the eigenvalue
$\lambda$ of $T$ is \emph{simple} if $\mult(\lambda, T)=1$. 

If~$E$ is also an algebra, for~$g \in E$, we denote by~$M_g$ the multiplication (possibly
unbounded) operator defined by~$M_g(h)=gh$ for all~$h \in E$.

\subsection{Invariance and continuity of the spectrum for compact operators}

We collect some known results on the spectrum and multiplicity of eigenvalues related to
compact operators. Let~$(E, \norm{\cdot})$ be a Banach space. Let~$A\in \cll(E)$. We
denote by $A^\top$ the adjoint of $A$. A sequence~$(A_n,n \in \N)$ of elements
of~$\cll(E)$ converges strongly to~$A \in \cll(E)$ if~$\lim_{n\rightarrow \infty }
\norm{A_nx -Ax}=0$ for all~$x\in E$. Following~\cite{anselone}, a set of
operators~$\ca \subset \cll(E)$ is \emph{collectively compact} if the set~$\{ A x \,
\colon \, A \in \ca, \, \norm{x}\leq 1 \}$ is relatively compact. Recall that the spectrum
of a compact operator is finite or countable and has at most one accumulation point, which
is $0$. Furthermore, $0$ belongs to the spectrum of compact operators in infinite
dimension. We refer to \cite{schaefer_banach_1974} for an introduction to Banach lattices
and positive operators; we shall only consider the Banach lattices 
$L^p(\Omega, \mu)$ for $p\geq 1$ on a probability space $(\Omega,
\cf, \mu)$ and a bounded operator $A$ is positive if $A(L^p_+)\subset L^p_+$. 

\begin{lemma}\label{lem:prop-spec-mult}
  Let $A, B$ be elements of $\cll(E)$. 
  \begin{propenum}
  \item\label{item:A-B} If $E$ is a Banach lattice, and if~$A$,~$B$ and~$A-B$ are positive
    operators, then we have:
    \begin{equation}\label{eq:r(A)r(B)} \rho(A)\geq \rho(B). 
    \end{equation}
  \item\label{item:spec-adjoint-mult} If $A$ is compact, then we have $AB$ and $BA$
    compact and:
    \begin{align}\label{eq:adjoint-mult}
      \spec(A)=\spec(A^\top) &\quad\text{and}\quad
      \mult(\lambda,A)=\mult(\lambda, A^\top) \quad\text{for $\lambda\in
      \C^*$},\\
      \label{eq:r(AB)-mult}
      \spec(AB)=\spec(BA & \quad\text{and}\quad
      \mult(\lambda, AB)=\mult(\lambda, BA)\quad\text{for $\lambda\in \C^*$},
    \end{align}
    and in particular:
    \begin{equation}\label{eq:r(AB)}
      \rho(AB)=\rho(BA). 
    \end{equation}
  \item \label{item:density-mult} Let~$(E', \norm{\cdot}')$ be a Banach space such
    that~$E'$ is continuously and densely embedded in~$E$. Assume that $A(E')\subset E'$,
    and denote by~$A'$ the restriction of $A$ to~$E'$ seen as an operator on~$E'$. If $A$
    and $A'$ are compact, then we have:
    \begin{equation}\label{eq:sT=sT'}
      \spec(A)=\spec(A')
      \quad\text{and}\quad
      \, \mult(\lambda, A)=\mult(\lambda, A')\quad\text{for $\lambda\in
      \C^*$}.
    \end{equation}
  \item \label{item:collectK-cv} Let~$(A_n, n\in \N)$ be a collectively compact sequence
    which converges strongly to~$A$. Then, we have $\lim_{n\rightarrow \infty }
    \spec(A_n)=\spec(A)$ in~$(\ck, d_\mathrm{H})$, $\lim_{n\rightarrow }
    \rho(T_n)=\rho(T)$ and for $\lambda\in
    \spec(A)\cap \C^*$, $r>0$ such that $\lambda'\in \spec(A)$ and
    $|\lambda-\lambda'|\leq r$ implies $\lambda=\lambda'$, and
    all $n$ large enough:
    \begin{equation}
      \label{eq:collectK-mult}
      \mult(\lambda, A)= \sum_{\lambda'\in
      \spec(A_n),\, |\lambda-\lambda'|\leq r} \mult(\lambda', A_n).
    \end{equation}
  \end{propenum}
\end{lemma}

\begin{proof}
  Property \ref{item:A-B} can be found in \cite[Theorem~4.2]{marek}. Equation
  \eqref{eq:adjoint-mult} from Property \ref{item:spec-adjoint-mult} can be deduced from 
  from \cite[Theorem~p.~20]{konig}. Using \cite[Proposition~p.~25]{konig}, we get the
  second part of \eqref{eq:r(AB)-mult} and $ \spec(AB) \cap \C^*=\spec(BA) \cap \C^*$, and
  thus \eqref{eq:r(AB)} holds. To get the first part of \eqref{eq:r(AB)-mult}, see
  \cite[Lemma~3.2]{ddz-theo}. 

  \medskip

  We now provide a short proof for Property \ref{item:density-mult}. According to
  \cite[Corollary~1 and Section~6]{halber}, we have $\spec(A)=\spec(A')$. Let $\lambda\in
  \spec(A) \cap \C^*$. Since the multiplicity of $\lambda$ for $A$ is finite, we get that
  $\mult(\lambda, A)=\dim \left(\ker (A -\lambda \mathrm{Id})^n\right)$ for $n$ large
  enough, and similarly for $\mult(\lambda, A')$. Clearly, we have $\ker (A' -\lambda
  \mathrm{Id})^n\subset \ker (A -\lambda \mathrm{Id})^n$. Let us prove that $\ker (A
  -\lambda \mathrm{Id})^n\subset \ker (A' -\lambda \mathrm{Id})^n$. Let $x\in \ker (A
  -\lambda \mathrm{Id})^n$ and $(x_\ell, \ell\in \N)$ be a sequence of elements of $E'$
  which converges (in $E$) towards $x$. Up to taking a sub-sequence, since $A'$ is
  compact, we can assume that $A'x_\ell$ converges in $E'$, say towards $y\in E'$. We
  deduce that:
  \begin{align*}
    \lambda^n x
    &=\sum_{k=1}^n \binom{n}{k} (-\lambda)^{n-k+1} A^k x\\
    &= \lim_{\ell\rightarrow \infty } \sum_{k=1}^n \binom{n}{k} (-\lambda)^{n-k+1} A^k x_\ell\\
    &= \lim_{\ell\rightarrow \infty } \sum_{k=1}^n \binom{n}{k} (-\lambda)^{n-k+1} (A')^{k-1} (A' x_\ell)\\
    &= \sum_{k=1}^n \binom{n}{k} (-\lambda)^{n-k+1} (A')^{k-1} y.
  \end{align*}
  Since $\lambda \neq 0$, we get that $x$ belongs to $E'$ and thus $(A' -\lambda
  \mathrm{Id})^n x= (A -\lambda
  \mathrm{Id})^nx=0$, that is
  $\ker (A -\lambda \mathrm{Id})^n\subset \ker (A' -\lambda
  \mathrm{Id})^n$. Then use the definition of the multiplicity to
  conclude.

  \medskip

  We eventually check Point \ref{item:collectK-cv}. We deduce from \cite[Theorems~4.8 and
  4.16]{anselone} (see also (d), (g) [take care that $d(\lambda, K)$ therein is the
  algebraic multiplicity of $\lambda$ for the compact operator $K$ and not the geometric
  multiplicity] and (e) in~\cite[Section~3]{SpectralProperAnselo1974})
  that~$\lim_{n\rightarrow \infty } \spec(T_n)=\spec(T)$ and \eqref{eq:collectK-mult}.
  Then use that the function~$\mathrm{rad}$ is continuous to deduce the convergence of the
  spectral radius from the convergence of the spectra.
\end{proof}

\subsection{Kernel operators}

We define a \emph{kernel} (resp. \emph{signed kernel}) on~$\Omega$ as a $\R_+$-valued
(resp. $\R$-valued) measurable function defined on~$(\Omega^2, \mathscr{F}^{\otimes 2})$.
For~$f,g$ two non-negative measurable functions defined on~$\Omega$ and~$\kk$ a kernel
on~$\Omega$, we denote by $f\kk g$ the kernel defined by:
\begin{equation}\label{eq:def-fkg}
  f\kk g:(x,y)\mapsto f(x)\, \kk(x,y) g(y).
\end{equation}

For~$p \in (1, +\infty )$, we define the double norm of a signed kernel~$\kk$ on $L^p$ by: 
\begin{equation}\label{eq:Lp-integ-cond}
  \norm{\kk}_{p,q}=\left(\int_\Omega\left( \int_\Omega \abs{\kk(x,y)}^q\,
  \mu(\mathrm{d}y)\right)^{p/q} \mu(\mathrm{d}x) \right)^{1/p}
  \quad\text{with~$q$ given by}\quad \frac{1}{p}+\frac{1}{q}=1.
\end{equation}
We say that $\kk$ has a finite double norm, if there exists~$p \in (1, +\infty )$ such
that $\norm{\kk}_{p,q}<+\infty$. To such a kernel $\kk$, we then associate the positive
integral operator~$T_\kk$ on~$L^p$ defined by:
\begin{equation}\label{eq:def-Tkk}
  T_\kk (g) (x) = \int_\Omega \kk (x,y)\, g(y)\,\mu(\mathrm{d}y)
  \quad \text{for } g\in L^p \text{ and } x\in \Omega.
\end{equation}
According to~\cite[p. 293]{grobler}, $T_\kk$ is compact. It is well known and easy to
check that:
\begin{equation}\label{eq:double-norm-norm}
  \norm{ T_\kk }_{L^p}\leq \norm{\kk}_{p,q}.
\end{equation}
We define the \emph{reproduction number} associated to the operator $T_\kk$ as:
\begin{equation}\label{eq:def-R0}
  R_0[\kk]=\rho(T_\kk).
\end{equation}

\medskip

The proof of the next stability result appears already in \cite{ddz-theo} (but for
\eqref{eq:cv-mult-vp} whose proof relies on \eqref{eq:collectK-mult} and is left to the
reader).

\begin{corollary}\label{cor:cv-kn}
  Let~$p\in (1, +\infty )$. Let~$(\kk_n, n\in \N)$ and~$\kk$ be kernels on~$\Omega$ with
  finite double norms on~$L^p$ such that~$\lim_{n\rightarrow\infty } \norm{\kk_n -\kk}_{p,q}=0$.
 Then, we have $\lim_{n\rightarrow \infty }
  \spec(T_{\kk_n})=\spec(T_\kk)$ in~$(\ck, d_\mathrm{H})$, $\lim_{n\rightarrow }
  \rho(T_{\kk_n})=\rho(T_\kk)$ and for $\lambda\in
    \spec(T_\kk)\cap \C^*$, $r>0$ such that $\lambda'\in \spec(T_\kk)$ and
    $|\lambda-\lambda'|\leq r$ implies $\lambda=\lambda'$, and
    all $n$ large enough:
\begin{equation}
   \label{eq:cv-mult-vp}
      \mult(\lambda, T_\kk)= \sum_{\lambda'\in
        \spec(T_{\kk_n}),\, |\lambda-\lambda'|\leq r} \mult(\lambda', T_{\kk_n}).
\end{equation}
\end{corollary}

\subsection{Irreducibility, quasi-irreducibility and monatomic kernel}

We first define \emph{irreducible} and \emph{monatomic} kernels. For $A, B\in \cf$, we
write $A\subset B$ a.s.\ if $\mu(B^c \cap A)=0$ and $A=B$ a.s.\ if $A\subset B$ a.s.\ and
$B\subset A$ a.s. For $A, B\in \cf$, $x\in \Omega$ and an integrable kernel $\kk$, we simply
write $\kk(x,A)=\int_{ A} \kk(x,y)\, \mu(\rd y)$, $ \kk(B,x)=\int_{ B} \kk(z,x)\, \mu(\rd
z)$ and:
\[
  \kk(B, A)=
  \int_{B \times A} \kk(z,y)\, \mu(\rd z) \mu(\rd y).
\]
A set $A\subset \cf$ is \emph{$\kk$-invariant}, or simply \emph{invariant} when there is
no ambiguity on the kernel $\kk$, if $\kk(A^c, A)=0$. In the epidemiological setting, the
set $A$ is invariant if the sub-population $A$ does not infect the sub-population $A^c$.
If $\kk$ is symmetric, then $A$ is invariant if and only if $A^c$ is invariant. \medskip

A kernel $\kk$ is \emph{irreducible} or \emph{connected} if any $\kk$-invariant set $A$ is
such that a.s. $A=\emptyset$ or a.s. $A=\Omega$. According to
\cite[Theorem~V.6.6]{schaefer_banach_1974}, if $\kk$ is an irreducible kernel with finite
double norm, then we have $R_0[\kk]>0$. If the kernel is positive a.s., then it is
irreducible. Following \cite[Definition~2.11]{bjr}, we say that a kernel is
\emph{quasi-irreducible} if $\kk$ restricted to $\{\kk \equiv 0\}^c$, with $\{\kk \equiv
0\}=\{x\in \Omega\, \colon\, \kk(x, \Omega) + \kk( \Omega, x)=0\}$, is irreducible. The
quasi-irreducible property was introduced for symmetric kernel; for general kernel one can
consider the following weaker property. A kernel $\kk$ is \emph{monatomic} if the
operator $T_\kk$ has a unique (up to a multiplicative constant) non-negative
eigenfunction. Intuitively, this corresponds to have only one irreducible component.
Formally, this is also equivalent to the following two properties:
\begin{propenum}
\item\label{item:oa} There exists a measurable subset $\oa\subset \Omega$, the irreducible
  component or \emph{atom} such that:
  \begin{itemize}
    \item $\mu(\oa)>0$ and the kernel $\kk$ restricted to $\Omega_a$ is irreducible.
    \item If a.s. $\oa^c\neq \emptyset$ then the restriction of $T_\kk$ to $\oa^c$ is
      quasi-nilpotent, that is, $R_e[\kk](\ind{\oa^c})=0$. 
  \end{itemize}
 
\item\label{item:oi} There exists a measurable subset $\oi\subset \oa^c$,
  ``the sub-population infected by'' $\oa$, such that:
  \begin{itemize}
\item The sets $\oa\cup \oi$ and $\oi$ are invariant. 
\item The set $\oi$ is the minimal set such that $\oa\cup \oi$ is
  invariant: if $A$ is invariant and $\oa\subset A$ then
    a.s. $\oi\subset A$. 
  \end{itemize}
\end{propenum} 
In the epidemiological setting, the sub-population $\oi$ can only infect
itself, and the sub-population $\oa$ infects only itself and $\oi$; the
set $\oa\cup\oi$ corresponds to the support of the endemic equilibrium
in the supercritical regime, see~\cite[Lemma~5.12]{ddz-theo}. We refer
to \cite{schwartz61} for further details on the decomposition of a
kernel on its irreducible components; in particular the sets $\oa$ and
$\oi$ are unique up to the a.s.\ equivalence. We represented in
Figure~\ref{fig:monatomic} a monatomic kernel and in
Figure~\ref{fig:quasi} a quasi-irreducible kernel; the set $\Omega$
being ``nicely ordered'' so that the representation of the kernels are
upper triangular.

\begin{rem}\label{rem:quasi-irr}
  Irreducible and quasi-irreducible kernels are also monatomic (take $\oa=\{\kk\equiv 0\}^
  c$ and $\oi=\emptyset$). If the kernel $\kk$ is monatomic and symmetric, then we get
  $\kk=\ind{\oa}\,\kk\, \ind{\oa}$ and thus the kernel $\kk$ is quasi-irreducible. 
  \medskip
   
  The notion of irreducibility of a kernel depends only on its support: the kernel $\kk$
  is irreducible (resp.\ quasi-irreducible, resp. monatomic) if and only if the kernel
  $\ind{\{\kk>0\}}$ is irreducible (resp.\ quasi-irreducible, resp. monatomic).
  Furthermore, if $\kk$ is monatomic, then the kernels $\kk$ and $\ind{\{\kk>0\}}$ have
  the same atom $\oa$ and the same set $\oi$ infected by $\oa$.
\end{rem}

\begin{figure}
  \begin{subfigure}[T]{.5\textwidth}
    \centering
    \includegraphics[page=1]{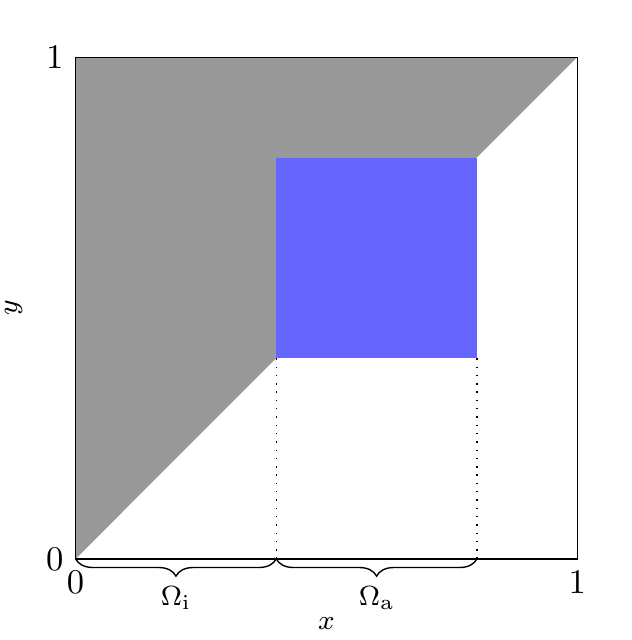}
    \caption{A representation of a monatomic kernel.}
    \label{fig:monatomic}
  \end{subfigure}%
  \begin{subfigure}[T]{.5\textwidth}
    \centering
    \includegraphics[page=2]{mono-atomic}
    \caption{A representation of a quasi-irreducible kernel.}
    \label{fig:quasi}
  \end{subfigure}
  \caption{Example of a monatomic and quasi-irreducible kernels
    $(x,y)\mapsto \kk(x,y)$, where $\kk(x,y)=0$ on the white zone and
    $\kk$ reduced to
    the blue zone is irreducible. }
  \label{fig:monatomic-quasi}
\end{figure}

The introduction of monatomic kernel is also motivated by the following result which can
be deduced from \cite[Theorem~V.6.6]{schaefer_banach_1974} and
\cite[Theorem~8]{schwartz61}, see also Section~\ref{sec:reducible}.

\begin{lemma}\label{lem:R0simple}
  Let $\kk$ be a kernel with finite double norm and set $R_0=R_0[\kk]$. If the kernel
  $\kk$ is monatomic then $R_0>0$ and $R_0$ is simple (\emph{i.e.} $\mult(R_0, T_\kk)=1$).
  If $R_0$ is simple and the only eigenvalue in $(0, +\infty )$, then the kernel $\kk$ is
  monatomic.
\end{lemma}

\subsection{The effective reproduction number~%
  \texorpdfstring{$R_e$}{Re}}\label{sec:weak-topo}

A \emph{vaccination strategy}~$\eta$ of a vaccine with perfect efficiency is an element
of~$\Delta$, where~$\eta(x)$ represents the proportion of \emph{\textbf{non-vaccinated}}
individuals with feature~$x$. In particular $\eta=\un$ (the constant function equal to
1) corresponds to no vaccination and $\eta=\zero$ (the constant function equal to 0)
corresponds to the whole population vaccinated. Notice that~$\eta\, \mathrm{d} \mu$
corresponds in a sense to the effective population. Let $\kk$ be a kernel on $\Omega$
with finite double norm on $L^p$. For~$\eta\in \Delta$, the kernel~$\kk \eta$ has also a
finite double norm on~$L^p$ and the operator~$M_\eta$ is bounded, so that the operator
$T_{\kk \eta} = T_\kk M_\eta$ is compact. We can define the \emph{effective spectrum}
function~$\spec[\kk]$ from~$\Delta$ to~$\ck$ by:
\begin{equation}\label{eq:def-sigma_e}
  \spec[\kk](\eta)=\spec(T_{\kk\eta}), 
\end{equation}
the \emph{effective reproduction number} function $R_e[\kk]=\mathrm{rad}\circ \spec[\kk]$
from~$\Delta$ to~$\R_+$ by:
\begin{equation}
  \label{eq:def-R_e}
  R_e[\kk](\eta)=\mathrm{rad}( \spec(T_{\kk \eta}))=\rho(T_{\kk\eta}),
\end{equation}
and the corresponding reproduction number is then given by $ R_0[\kk]=R_e[\kk](\un)$. 
When there is no risk of confusion on the kernel~$\kk$, we simply write $R_e$ and $R_0$
for the function~$R_e[\kk]$ and the number~$R_0[\kk]$.

We can see~$\Delta$ as a subset of~$L^1 $, and consider the corresponding \emph{weak
topology}: a sequence~$(g_n, n \in \N)$ of elements of~$\Delta$ converges weakly to~$g$ if
for all~$h \in L^\infty $ we have:
\begin{equation}
  \label{eq:weak-cv}
  \lim\limits_{n \to \infty} \int_\Omega h g_n \, \mathrm{d}\mu= \int_\Omega h
  g\, \mathrm{d}\mu.
\end{equation}
Notice that \eqref{eq:weak-cv} can easily be extended to any function $h\in L^q$ for any
$q\in (1, +\infty )$; so that the weak-topology on $\Delta$, seen as a subset of $L^p$
with $1/p+1/q=1$, can be seen as the trace on $\Delta$ of the weak topology on $L^p$. From
the Banach-Alaoglu theorem, we get that the set~$\Delta$ endowed with the weak topology is
compact and sequentially compact, see \cite[Lemma~3.1]{ddz-theo}. 

We also recall the properties of the effective reproduction number given
in \cite[Proposition~4.1 and Theorem~4.2]{ddz-theo}. 

\begin{proposition}\label{prop:R_e}
 Let $\kk$ be a kernel on a probability space $(\Omega, \cf \mu)$ with finite double
  norm. Then, the functions~$\spec[\kk]$
  and~$R_e=R_e[\kk]$ are continuous functions from $\Delta$ respectively
  to~$\ck$ (endowed with the Hausdorff distance) and
  to~$\R_+$. Furthermore, the function~$R_e=R_e[\kk]$ satisfies the
  following properties:
  \begin{propenum}
  \item\label{prop:a.s.-Re}%
    $R_e(\eta_1)=R_e(\eta_2)$ if~$\eta_1=\eta_2,\, \mu\as$, and~$\eta_1, \eta_2\in
    \Delta$,
  \item\label{prop:min_Re}%
    $R_e(\zero) = 0$ and~$R_e(\un) = R_0$,
  \item\label{prop:increase}%
    $R_e(\eta_1) \leq R_e(\eta_2)$ for all~$\eta_1, \eta_2\in \Delta$ such
    that~$\eta_1\leq \eta_2$,
  \item\label{prop:normal}%
    $R_e(\lambda \eta) = \lambda R_e(\eta)$, for all~$\eta \in \Delta$ and~$\lambda \in
    [0,1]$.
  \end{propenum}
\end{proposition}

We complete Corollary~\ref{cor:cv-kn} on the stability property of the spectrum and
spectral radius with respect to the kernel~$\kk$, see \cite[Proposition~4.3]{ddz-theo}. 

\begin{proposition}[Stability of~$\grR$ and~$\grS$]\label{prop:Re-stab}
  Let~$p\in (1, +\infty )$. Let~$(\kk_n, n\in \N)$ and~$\kk$ be kernels on~$\Omega$ with
  finite double norms on~$L^p$. If~$\lim_{n\rightarrow\infty} \norm{\kk_n - \kk}_{p,q}=0$,
  then we have:
  \begin{equation}
    \label{eq:Re-stab}
    \lim_{n\rightarrow\infty }\, \sup_{\eta\in \Delta} \Big|R_e[\kk_n](\eta) -
    R_e[\kk](\eta)\Big|=0
    \quad\text{and}\quad
    \lim_{n\rightarrow\infty }\, \sup_{\eta\in \Delta} d_\mathrm{H}\Big (
    \spec[\kk_n](\eta), \spec [\kk](\eta)\Big)=0.
  \end{equation}
\end{proposition}

\section{Spectrum-preserving transformations}\label{sec:equivalent}

In this section, we consider a given probability state space $(\Omega, \cf, \mu)$, and we
discuss two operations on the kernel $\kk$ that leave the functions $\spec[\kk]$ and
$R_e[\kk]$ defined on $\Delta$. Recall the convention~\eqref{eq:def-fkg} for the kernel $f
\kk g$ defined from the kernel $\kk$ and the non-negative functions $f$ and $g$.

\begin{lemma}\label{lem:hk/h}
  Let $\kk$ be a kernel on $\Omega$ and $h$ be a non-negative measurable function on
  $\Omega$. 
  \begin{propenum}
  \item\label{lem:hk=kh} If $h\kk$ and $\kk h$ have finite double norms (with possibly
    different $p$), then we have:
    \begin{align*}
      \spec[h\kk]= \spec[h\kk\mathds{1}_{\set{h>0}}] &=
      \spec[\mathds{1}_{\set{h>0}}\kk h] = \spec[\kk h],\\
      R_e[h\kk]=R_e[h \kk \mathds{1}_{\set{h>0}}] &=
      R_e[\mathds{1}_{\set{h>0}} \kk h]=R_e[\kk h]. 
    \end{align*}
  \item\label{lem:k=hk/h} If $h$ is positive and if $\kk$ and $h\kk/h$ have finite double
    norms (with possibly different $p$), then we have:
    \[ \spec[\kk]= \spec[h\kk/h] \quad\text{and}\quad R_e[\kk]=R_e[h\kk/h]. \]
  \item\label{lem:k=ktop} If $\kk$ and its transpose $\kk^\top:(x,y) \mapsto \kk(y,x)$ 
    have finite double norms (with possibly different $p$), then we have:
    \[ \spec[\kk]=\spec[\kk^\top] \quad \text{and} \quad R_e[\kk] = R_e[\kk^\top]. \]
  \end{propenum} 
\end{lemma}

Even if \ref{lem:k=hk/h} is a consequence of \ref{lem:hk=kh}, we state it separately
since~\ref{lem:k=hk/h} and~\ref{lem:k=ktop} describe two modifications of~$\kk$ that leave
the functions~$R_e$ and $\spec$ invariant. See Equation~\eqref{eq:kk=tilde-kk} for an
other transformation on the kernels which leaves the functions $R_e$ and $\spec$
invariant. See also \cite{ddz-pres} for further results in the finite dimensional case. 
 
\begin{proof}
  Since $R_e=\mathrm{rad}\circ \spec$, we only need to prove
  \ref{lem:hk=kh}-\ref{lem:k=ktop} for the function $\spec$. 
  We give the detailed proof of~\ref{lem:k=hk/h} and leave the proof
  of~\ref{lem:hk=kh}, which is very similar, to the reader. We first
  assume that $\kk$, $h$ and $1/h$ are bounded. The operators
  $T_{\kk \eta}$ and $T_{h\kk \eta /h}$ and the multiplication operators
  $M_h$ and $M_{1/h}$ are bounded operators on $L^p$ for
  $p\in (1, +\infty )$. We have, using that
  $T_{\kk \eta /h} = T_\kk \, M_{\eta/h}$ is compact
  and~\eqref{eq:r(AB)-mult} for the second equality:
  \[
    \spec(T_{\kk \eta}) = \spec (T_{\kk \eta/h}\, M_h) = \spec(M_h T_{\kk \eta/h})=
    \spec(T_{h\kk \eta/h}).
  \]
  Since $\eta \in \Delta$ is arbitrary, this gives that $\spec[\kk]=\spec[h\kk /h]$.

  In the general case, we use an approximation scheme. Define the kernel $\kk_n= (v_n\kk
  v_n) \wedge n$ with $v_n=\mathds{1}_{\{n\geq h\geq 1/n\}}$ and the function $h_n=n^{-1}
  \vee (h\wedge n)$ for $n\in \N^*$. From the first part of the proof, we get
  $\spec[\kk_n]=\spec[\kk'_n]$, with $\kk'_n=h_n \kk_n/h_n$. Since $\norm{\kk}_{p,q}$ is
  finite for some $p\in (1, +\infty )$, we get by dominated convergence that
  $\lim_{n\rightarrow \infty } \norm{\kk- \kk_n}_{p,q}=0$, and we deduce from Proposition
  \ref{prop:Re-stab} that $\lim_{n\rightarrow \infty } \spec[\kk_n]=\spec[\kk]$.
  Similarly, setting $\kk'=h\kk/h$, the norm $\norm{\kk'}_{p',q'}$ is finite for some
  $p'\in (1, +\infty )$, and thus $\lim_{n\rightarrow \infty } \norm{\kk'-
  \kk'_n}_{p',q'}=0$, so that $\lim_{n\rightarrow \infty } \spec[\kk'_n]=\spec[\kk']$.
  This proves that $\spec[\kk]=\spec[\kk']$, and thus~\ref{lem:k=hk/h}.

 \medskip

  We now prove \ref{lem:k=ktop}.
  For any $\eta\in \Delta$, the kernel 
  $\kk^\top\eta$ defines a bounded integral operator in $L^q$, whose
  adjoint is $T_{\eta\kk}$. Since 
  the spectrum of an operator and its adjoint are the same, we get $\spec[\kk^\top](\eta)
  = \spec(T_{\kk^\top \eta}) = \spec(T_{\eta \kk}) = \spec(M_\eta T_\kk )=\spec(T_\kk
  M_\eta)= \spec[\kk](\eta)$, where the fourth equality follows once more
  from~\eqref{eq:r(AB)-mult}. Since this is true for any $\eta\in \Delta$, this gives
  $\spec[\kk^\top] = \spec[\kk]$.
\end{proof}

\begin{remark}\label{rem:model=g(-1)k}
  In the infinite dimensional SIS model developed in
  \cite{delmas_infinite-dimensional_2020}, the next generation operator is given by the
  integral operator $T_\kk$, where the kernel~$\kk=k/\gamma$ is defined in terms of a
  transmission rate kernel $k(x,y)$ and a recovery rate function~$\gamma$ by the
  product~$\kk(x,y)=k(x,y)/\gamma(y)$; and the reproduction number~$R_0$ is then the
  spectral radius $\rho(T_\kk)$ of~$T_\kk$. Furthermore the operator $T_{\gamma^{-1} k}$
  appears very naturally in the definition of the maximal equilibrium $\mathfrak{g}$ which
  is solution to \cite[Equation~(24)]{delmas_infinite-dimensional_2020}, that is 
  $T_{\gamma^{-1} k} (\mathfrak{g})=\mathfrak{g}/(1-\mathfrak{g})$. According to
  Lemma~\ref{lem:hk/h}~\ref{lem:hk=kh}, provided that $k/\gamma$ and $\gamma^{-1} k$ have
  finite double norms, the next generation operator $T_{k/\gamma}$ and $T_{\gamma^{-1} k}$
  have the same effective spectrum function.
\end{remark}

We shall use the following extension in the proof of Lemma~\ref{lem:p>p}.

\begin{remark}\label{rem:hk/h-mult}
  Following closely the proof of Lemma~\ref{lem:hk/h}~\ref{lem:k=hk/h} and using
  Corollary~\ref{cor:cv-kn}, we also get that if $h$ is positive and if $\kk$ and $h\kk/h$
  have finite double norms (with possibly different $p$), then we have:
  \begin{equation}\label{eq:mult:k=hk/h}
    \mult(\lambda, T_{\kk})=\mult (\lambda, T_{h\kk/h}) \quad\text{ for all $\lambda\in
    \C^*$.}
  \end{equation}
\end{remark}

\section{Sufficient conditions for convexity or concavity
  of~\texorpdfstring{$R_e$}{Re}}\label{sec:Hill_Longini}

\subsection{A conjecture from Hill and Longini}
\label{sec:comparison_Hill_Longini}

Recall that, in the metapopulation framework, the effective reproduction number is equal
to the spectral radius of the matrix $K\cdot\diag(\eta)$, where $K$ has non-negative
entries and is the next-generation matrix and $\eta$ is the vaccination strategy giving
the proportion of non-vaccinated people in each groups. The Hill-Longini conjecture
appears in \cite{hill-longini-2003} and gives conditions on the spectrum of the
next-generation matrix that implies the convexity or the concavity of the effective
reproduction number. It states that the function $R_e[K]$ is:
\begin{propenum}
\item\label{item:cvx-K} convex when $\spec(K) \subset \R_+$,
\item\label{item:cave-K} concave when $\spec(K) \backslash \{ R_0 \} \subset \R_-$.
\end{propenum}

It turns out that the conjecture cannot be true without additional assumption on the
matrix $K$. Indeed, consider the following next-generation matrix:
\begin{equation}\label{eq:next-gen-counter-conv}
  K =
  \begin{pmatrix}
    16 & 12 & 11 \\
    1 & 12 & 12 \\
    8 & 1 & 1
  \end{pmatrix}
\end{equation}
Its eigenvalues are approximately equal to $24.8$, $2.9$ and $1.3$. Since $R_e$ is
homogeneous, the function is entirely determined by the value it takes on the plane $\{
\eta \, \colon \, \eta_1 + \eta_2 + \eta_3 = 1/3 \}$. The graph of the function $R_e$
restricted to this set has been represented in Figure~\ref{fig:counter-graph-conv}. The
view clearly shows the saddle nature of the surface. Hence, the Hill-Longini conjecture
\ref{item:cvx-K} is contradicted in its original formulation. In
Figure~\ref{fig:counter-kernel-conv}, we have represented the corresponding kernel model
when the population is split equally into three groups, \emph{i.e.}, $\mu_1 = \mu_2 =
\mu_3 = 1/3$.

\begin{figure}
  \begin{subfigure}[T]{.5\textwidth} \centering
    \includegraphics[page=1]{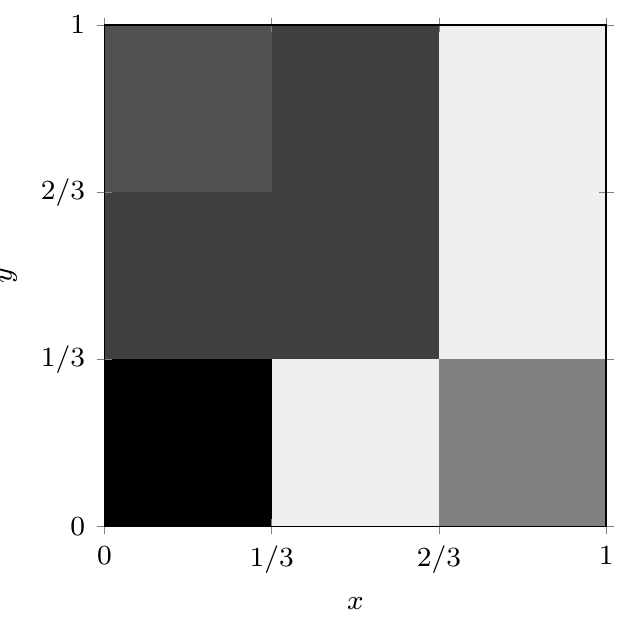}
    \caption{Grayplot of the kernel corresponding to the next-generation
    matrix \eqref{eq:next-gen-counter-conv} when the sub-populations have same size.}
    \label{fig:counter-kernel-conv} 
  \end{subfigure}%
  \begin{subfigure}[T]{.5\textwidth} \centering
    \includegraphics[page=3]{counter-examples}
    \caption{The plan of strategies $\eta$ with cost
    $\costu(\eta)=2/3$ is represented as a gray surface. The triangulated surface corresponds to
  the graph of $R_e$ restricted to this set.}
    \label{fig:counter-graph-conv}
  \end{subfigure}%
  \caption{Counter-example of the Hill-Longini conjecture (convex
    case).}
  \label{fig:counter-conv}
\end{figure}

\medskip

In the same manner, the eigenvalues of the following next-generation matrix:
\begin{equation}\label{eq:next-gen-counter-conc}
  K = 
  \begin{pmatrix}
    9 & 13 & 14 \\
    18 & 6 & 5 \\
    1 & 6 & 6
  \end{pmatrix}
\end{equation}
are approximately equal to $26.3$, $-1.4$ and $-3.9$. Thus, $K$ satisfies the condition
that should imply the concavity of the effective reproduction number in the Hill-Longini
conjecture~\ref{item:cave-K}. However, as we can see in
Figure~\ref{fig:counter-graph-conc}, the function $R_e$ is neither convex nor concave. In
Figure~\ref{fig:counter-kernel-conc}, we have represented the corresponding kernel model
when the population is splitted equally into three groups, \emph{i.e.}, $\mu_1 = \mu_2 =
\mu_3 = 1/3$.

\begin{figure}
  \begin{subfigure}[T]{.5\textwidth} \centering
    \includegraphics[page=2]{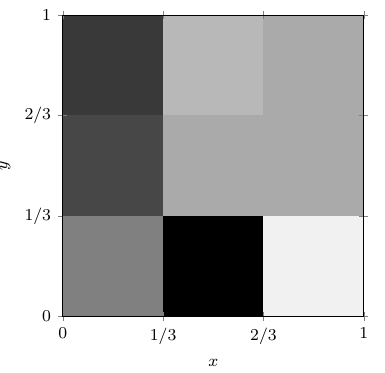}
    \caption{Grayplot of the kernel corresponding to the next-generation
    matrix \eqref{eq:next-gen-counter-conc} when the sub-populations have same size.}
  \label{fig:counter-kernel-conc}
  \end{subfigure}%
  \begin{subfigure}[T]{.5\textwidth} \centering
    \includegraphics[page=4]{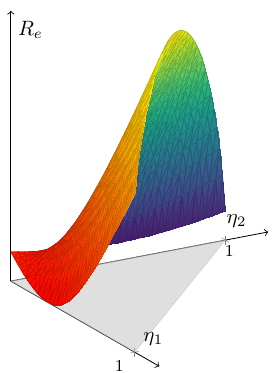}
    \caption{The plan of strategies $\eta$ with cost
    $\costu(\eta)=2/3$ is represented as a gray surface. The triangulated surface corresponds to
  the graph of $R_e$ restricted to this set.}
  \label{fig:counter-graph-conc}
  \end{subfigure}%
  \caption{Counter-example of the Hill-Longini conjecture (concave case).}
  \label{fig:counter-conc}
\end{figure}

\medskip

Despite these counter-examples, the Hill-Longini conjecture is indeed true when making
further assumption on the next-generation matrix. Let~$M$ be a square real matrix. The
matrix~$M$ is \emph{diagonally similar} to a matrix~$M'$ if there exists a non singular
real diagonal matrix~$D$ such that~$M=D\cdot M' \cdot D^{-1}$. The matrix~$M$ is said to
be \emph{diagonally symmetrizable} or simply \emph{symmetrizable} if it is diagonally
similar to a symmetric matrix, or, equivalently, if $M$ admits a decomposition~$M=D \cdot
S$ (or~$M=S\cdot D$), where~$D$ is a diagonal matrix with positive diagonal entries
and~$S$ is a symmetric matrix. If a matrix $M$ is diagonally symmetrizable, then its
eigenvalues are real since similar matrices share the same spectrum. We obtained the
following result when the next-generation matrix is symmetrizable.

\begin{theorem}\label{th:hill-longini-meta}
  Suppose the non-negative matrix $K$ is diagonally symmetrizable. 
  \begin{propenum}
  \item\label{pt:conv-meta} If $\spec(K) \subset \R_+$, then the
    function $R_e[K]$ is convex.
  \item If $R_0$ is a simple eigenvalue of $K$ and
    $\spec(K) \subset \R_-\cup \{ R_0 \} $, then the function
    $R_e[K]$ is concave.
  \end{propenum}
\end{theorem}

This result is a particular case of Theorem~\ref{th:hill-longini} below. The first
point~\ref{pt:conv-meta} has been proved by Cairns in \cite{EpidemicsInHeCairns1989}. In
\cite{Fried80}, Friedland obtained that, if the next-generation matrix $K$ is not singular
and if its inverse is an M-matrix (\emph{i.e.}, its non-diagonal coefficients are
non-positive), then $R_e$ is convex. Friedland's condition does not imply that $K$ is
symmetrizable nor that $\spec(K) \subset \R_+$. On the other hand, the following matrix is
symmetric definite positive (and thus $R_e$ is convex) but its inverse is not an M-matrix.
\begin{equation*}
  K = 
  \begin{pmatrix}
    3 & 2 & 0 \\
    2 & 2 & 1 \\
    0 & 1 & 4
  \end{pmatrix}
  \qquad \text{with inverse} \quad
  K^{-1} = 
  \begin{pmatrix}
    1.4 & -1.6 & 0.4 \\
    -1.6 & 2.4 & -0.6 \\
    0.4 & - 0.6 & 0.4 
  \end{pmatrix}.
\end{equation*}
Thus Friedland's condition and Property~\ref{pt:conv-meta} in
Theorem~\ref{th:hill-longini-meta} are not comparable. Note that if~$K$ is symmetrizable
and its inverse is an M-matrix, then the eigenvalues of $K$ are actually non-negative
thanks to \cite[Chapter~6~Theorem~2.3]{NonnegativeMatBerman1994} and one can apply
Theorem~\ref{th:hill-longini-meta}~\ref{pt:conv-meta}.

\subsection{Generalization for the kernel model}

In this section, we give the analogue of Theorem~\ref{th:hill-longini-meta} for kernels
instead of matrices. First, we proceed with some definitions.

We say that a kernel~$\kk'$ is an Hilbert-Schmidt non-negative symmetric kernel if
$\kk'\geq 0$, $\norm{\kk'}_{2,2}<+\infty$ and $\mu(\rd x)\otimes \mu(\rd y)$-a.e.
$\kk'(x,y)=\kk'(y,x)$. By analogy with the matrix case and following \cite[Example~A,
p252]{zaanen_linear}, we introduce the notion of symmetrizability in the context of
kernels.

\begin{definition}[Diagonally HS kernel]\label{def:symmetrizable}
  A kernel~$\kk$ on~$\Omega$ is diagonally HS if there exists an Hilbert-Schmidt symmetric
  non-negative kernel~$\kk'$ on~$\Omega$ and two positive measurable functions~$f,g$
  defined on~$\Omega$ such that~$ \kk=f \kk' g$ a.s., that is $\mu(\rd x)\otimes \mu(\rd
  y)$-a.s.:
  \begin{equation}
    \label{eq:rep:k=fkg}
    \kk(x,y)=f(x)\, \kk'(x,y)\, g(y).
  \end{equation}
  If furthermore $f$ and $g$ are bounded and bounded away from $0$, then we say that the
  kernel~$\kk$ is strongly diagonally HS. 
\end{definition}

The notion of diagonally HS kernel appears naturally when considering the SIS model on
graphons; see \cite[Example~1.3]{delmas_infinite-dimensional_2020}, where the kernel $\kk$
is written as $\kk= \beta W \theta$, where $\beta(x)$ represents the susceptibility and
$\theta(x)$ the infectiousness of the individuals with feature $x$, and $W$ models the
graph of the contacts within the population with the quantity $W(x, y)=W(y,x) \in [0, 1]$
representing the density of contacts between individuals with features $x$ and $y$.

\begin{remark}\label{rem:stronglyHS}
  We complete the notion of diagonally HS kernel with three comments. 
  \begin{propenum}
  \item In finite dimension (\emph{i.e.} $\Omega$ finite), a diagonally HS kernel is
    strongly diagonally HS.
  \item Notice that a strongly diagonally HS kernel has finite double norm in $L^2$. 
  \item\label{pt:non-negative-pos} Consider the decomposition \eqref{eq:rep:k=fkg}, where
    $f$ and $g$ are assumed to be non-negative instead of positive, with the other
    assumptions unchanged. Then using Lemma~\ref{lem:hk/h}~\ref{lem:hk=kh} and assuming
    that $\kk$ in~\eqref{eq:rep:k=fkg} has a finite double norm, we get that
    $R_e[\kk]=R_e[\kk_0]$ coincide on $\Delta$, where $\kk_0=\ind{\{fg>0\}} \,\kk\,
    \ind{\{fg>0\}}$. As $\kk_0=f' \, \kk'_0 \, g'$ with $\kk'_0=\ind{\{fg>0\}} \,\kk'\,
    \ind{\{fg>0\}}$ and $h'=h+\ind{\{fg=0\}}$ for $h\in \{f, g\}$, we get that the kernel
    $\kk_0$ is diagonally HS (indeed $f'$ and $g'$ are positive, and the other assumptions
    hold). So, as far as the study of $R_e[\kk]$ is concerned, without loss of generality
    one can indeed assume that the functions $f$ and $g$ which appear in the decomposition
    of a diagonally HS kernel are positive instead of non-negative.
  \end{propenum}
\end{remark}

The following elementary lemma states that the integral operator of a diagonally HS kernel
has real eigenvalues. 

\begin{lemma}\label{lem:Tk-real}
  Let $\kk$ be a diagonally HS kernel with finite double norm. The spectrum of $T_\kk$ is
  real, that is, $\spec(T_\kk)\subset \R$. 
\end{lemma}

\begin{proof}
  Let $\kk'$, $f$ and $g$ as in \eqref{eq:rep:k=fkg} and for $n\in \N^*$ set:
  \begin{equation}
    \label{eq:def-vn} v_n=\ind{\{n\geq f \geq 1/n \, \text{ and }\, n\geq g\geq 1/n\}}.
  \end{equation}
  Let $p\in (1, +\infty )$ be such that $\norm{\kk}_{p,q}$ is finite. By monotone
  convergence, we have $\lim_{n\rightarrow \infty } \norm{\kk - fv_n \kk' v_n g}_{p,q}=0$.
  We deduce that:
  \[ \spec(T_\kk)
    = \spec(T_{f\kk ' g})
    =\lim_{n\rightarrow \infty } \spec(T_{fv_n\kk'v_n g})
    =\lim_{n\rightarrow \infty } \spec(T_{\sqrt{fg}\, v_n\kk'v_n \sqrt{fg}}),
  \]
  where we used \eqref{eq:rep:k=fkg} for the first equality, Corollary~\ref{cor:cv-kn} for
  the second, Lemma~\ref{lem:hk/h}~\ref{lem:k=hk/h} with $h=v_n\, \sqrt{g/f} + (1-v_n)$ for
  the last. Since the kernel $\sqrt{fg}\, v_n\kk'v_n \sqrt{fg}$ is symmetric with finite
  double norm in $L^2$, we deduce that the associated compact integral operator is
  self-adjoint, and thus $\spec(T_{\sqrt{fg}\, v_n\kk'v_n \sqrt{fg}}) \subset \R$. Then,
  use that $\R$ is closed for the Hausdorff distance to deduce that $\spec(T_\kk) \subset
  \R$.
\end{proof}

For a compact operator $T$, we denote by $\pos(T)$ and $\nega(T)$ the number of
its positive and negative eigenvalues with their multiplicity:
\[
  \pos(T)= \sum_{\lambda>0} \mult(\lambda,T)
  \quad\text{and}\quad
  \nega(T)= \sum_{\lambda<0} \mult(\lambda,T). 
\]
Note that $R_0[\kk]>0$ implies that $\pos(T_\kk)\geq 1$. 

The following result is the analogue of Theorem~\ref{th:hill-longini-meta} for the kernel
model.

\begin{theorem} [Convexity/Concavity of $R_e$]\label{th:hill-longini}
 Let $\kk$ be a
strongly diagonally HS kernel. We consider the function $R_e=R_e[\kk]$ defined on
  $\Delta$.
\begin{propenum}
   \item\label{item:hill-longini-cvx} If $\nega(T_\kk)=0$, then the function $R_e$
     is convex.
   \item\label{item:hill-longini-cave} If
     $\pos(T_\kk)=1$, 
     then the function $R_e$ is concave.
   \end{propenum}
 \end{theorem}

In the case of diagonally HS kernels, we have the following partial
result.

\begin{proposition}\label{prop:hill-longini}
  Let $\kk$ be a diagonally HS kernel of finite double norm, with the HS kernel $\kk'$
  from \eqref{eq:rep:k=fkg}. We consider the function $R_e=R_e[\kk]$ defined on $\Delta$.
  \begin{propenum}
  \item\label{item:prop-hill-longini-cvx} If $\nega(T_{\kk'})=0$, then $\nega(T_\kk)=0$
    and the function $R_e$ is convex.
  \item\label{item:prop-hill-longini-cave} If $\pos(T_{\kk'})=1$, then $\pos(T_\kk)=1$
    and the function $R_e$ is concave.
  \end{propenum}
\end{proposition}

The proof for HS kernels is given in Section \ref{sec:convexe} for the convex case and in
Section \ref{sec:concave} for the concave case; the latter relies on the Sylvester's
inertia theorem which is presented in Section~\ref{sec:sylvester}. The extension to
(strongly) diagonally HS kernel follows from Sections~\ref{sec:proof-HL}.
 
\begin{remark}[Concavity and monatomicity]\label{rem:hill-longini}
  We assume $R_0>0$ with $R_0=R_0[\kk]$.
  \begin{propenum}
  \item If $R_e[\kk]$ is concave, then $\kk$ is monatomic, see Lemma~\ref{lem:conc-mono}.
  \item If $\kk$ is a strongly diagonally HS kernel, then the condition $\pos(T_\kk)=1$ in
    Theorem~\ref{th:hill-longini}~\ref{item:hill-longini-cave} implies that $R_0$ is
    simple and $\spec(T_\kk) \subset \R_-\cup\{R_0\}$, and thus $\kk$ is monatomic, see
    Lemma~\ref{lem:R0simple}.
  \item More generally, using the decomposition of a reducible kernel from Lemma
    \ref{lem:Rei}, we get that if $\spec(T_\kk) \subset \R_-\cup\{R_0\}$ and $\kk$ is a
    strongly diagonally HS kernel, then the function $R_e$ is the maximum of $m=\mult(R_0,
    T_\kk)$ concave functions which are non-zero on $m$ pairwise disjoint subsets of
    $\Delta$.
  \end{propenum}
\end{remark}

\begin{remark}
   It is unclear whether or not $\pos(T_\kk)=1$ (resp. $\nega(T_\kk)=0$) in
   Proposition~\ref{prop:hill-longini} implies that $\pos(T_{\kk'})=1$
   (resp. $\nega(T_{\kk'})=0$). 
\end{remark}

\begin{remark}\label{rem:config}
  A \emph{configuration model} corresponds in finite dimension to the next generation
  matrix having rank one, this is the so-called \emph{proportionate mixing} model in the
  metapopulation literature; see Cairns \cite{EpidemicsInHeCairns1989} for optimal
  vaccinations strategies in this setting.

  Motivated by the finite dimensional case, we say that a kernel $\kk$ is a
  \emph{configuration} kernel if there exist $p \in (1, +\infty)$, $f \in L^p$ and $g \in
  L^q$ where $q = p/(p-1)$ such that $\kk(x,y) = f(x) g(y)$, $\mu\otimes\mu$-almost
  surely. We also suppose that $\mu(fg > 0) > 0$. Such a kernel has finite double norm, as
  $\norm{\kk}_{p,q}=\norm{f}_p \norm{g}_q$. Following
  Remark~\ref{rem:stronglyHS}~\ref{pt:non-negative-pos}, we have $R_e[\kk]=R_e[\ind{\{fg >
  0\}} \,\kk \,\ind{\{fg > 0\}}]$ with $\ind{\{fg > 0\}} \,\kk \,\ind{\{fg > 0\}}$
  diagonally HS as $\ind{\{fg > 0\}} \,\kk \,\ind{\{fg > 0\}} = (f + \ind{\{f = 0\}})
  \ind{\{fg > 0\}} \ind{\{fg > 0\}} (g + \ind{\{g=0\}})$. Besides, the only eigenvalue of
  the kernel $\ind{\{fg > 0\}}(x) \ind{\{fg > 0\}}(y)$ different from $0$ is its spectral
  radius equal to $\mu(fg > 0)$ and it has multiplicity $1$. Applying
  Proposition~\ref{prop:hill-longini}, we obtain that $R_e$ is convex and concave and thus
  linear. This can be checked directly as:
  \begin{equation}
    R_e[\kk](\eta) = \int_\Omega f g\, \eta \, \mathrm{d}\mu.
  \end{equation}
  We shall provide in \cite{ddz-mono} a deeper study of configuration kernels in the
  context of epidemiology.
\end{remark}

\subsection{Sylvester's inertia theorem}
\label{sec:sylvester}

Following \cite[Section~4.1.2]{pelinovsky}, we state and provide a short proof for the
Sylvester's inertia theorem in our context; see also \cite[Theorem~4.5.8]{horn2012matrix}
in finite dimension. This result will be used to prove the concavity of $R_e$. 

\begin{theorem}[Sylvester's inertia theorem]\label{lem:sylvester}
  Let $(\Omega, \cf, \mu)$ be a probability space. Let $T'$ be a self-adjoint compact
  operator on $L^2( \mu)$, and two non-negative measurable functions~$f,g$ defined
  on~$\Omega$ which are bounded and bounded away from $0$. Set $T=M_f T' M_g$. Then, we
  have $\spec(T)\subset \R$ and:
  \begin{equation}
    \label{eq:sylvester}
    \pos(T)=\pos(T')
    \quad\text{as well as}\quad 
    \nega(T)=\nega(T').
  \end{equation}
\end{theorem}

\begin{proof}
  Set~$h=\sqrt{f/g}$, $M=M_{\sqrt{fg}}$ and
  \[
    T''=M T' M,
  \]
  so that $ T=M_h T'' M_{1/h}$.
   Thanks to \eqref{eq:r(AB)-mult}, we get that $\mult(\lambda,
   T)=\mult(\lambda, T'')$ for all $\lambda\in \C^*$.
   So, we need to prove \eqref{eq:sylvester} with $T$ replaced by
   $T''$. We only consider the number of positive eigenvalues as the number of
  negative eigenvalues can be handled similarly. 
  \medskip
   
  We introduce some general notations. For a self-adjoint compact
  operator $S$ on $L^2(\mu)$, let $(u_i, i\in I)$, with $I$ at most
  countable and $\sharp I=\pos(S)$, be a sequence of orthogonal
  eigenvectors associated to the positive eigenvalues
  $(\lambda_i, i\in I)$ of $S$. Let $U\subset L^2(\mu)$ be the (closed)
  vector sub-space spanned by $(u_i, i\in I)$. The orthogonal complement
  of $U$, say $U^\top$ is the (closed) vector space spanned by the
  kernel of $S$ and the eigenvectors associated to the negative
  eigenvalues. We consider the quadratic form $Q_S$ on $L^2(\mu)$
  defined by:
   \[
     Q_S(u)=\langle u,S u \rangle.
   \]
   Let $P_S$ be the orthogonal projection on $U^\top$. By decomposing
   $u$ on $U\oplus U^\top$, we get:
   \[
     Q_S(u)=\sum_{i\in I} \lambda_i \langle u, u_i
   \rangle^2 +Q_S(P_S(u)),
 \]
 and the quadratic form $Q_S \circ P_S$ is negative semi-definite. 
 \medskip

 We shall now prove that $\pos(T'')=\pos(T')$ by contradiction. First assume that
 $\pos(T')<\pos(T'')$, so in particular $\pos(T')$ is finite. Let
 $(u''_i, i\in I'')$ be a sequence of orthogonal eigenvectors
 associated to the positive eigenvalues $(\lambda''_i, i\in I'')$ of
 $T''$. Set $v_i=M u''_i$ for $i\in I''$. In particular, the
 dimension of the space spanned by $(v_i, i\in I'')$, which is equal to
 $\pos(T'')$, is larger than the finite dimension of the space $U$ spanned
 by the orthogonal eigenvectors $(u'_i, i\in I')$ associated to the
 positive eigenvalues of $T'$.
 Thus, solving a linear system, we get there exists $(c_i, i\in I'')$ such
 that $c_i\neq 0$ for at most $\pos(T')+1$ indices,
 $u=\sum_{i\in I''} c_i v_i\neq 0$, and $u \in U^\top$. 
On one hand,
 since $Q_{T'}$ is negative semi-definite on $U^\top$, we get
 $Q_{T'} (u)\leq 0$. On the other hand, we have:
 \[
   Q_{T'}(u)=\langle u, T'u \rangle
   = \sum_{i, j \in I''} c_i c_j\, \langle v_i, T' v_j \rangle
   = \sum_{i, j \in I''} c_i c_j\, \langle u''_i, T'' u''_j \rangle
   = \sum_{i\in I''} c_i^2 \lambda''_i>0.
 \]
 By contradiction, we deduce that $\pos(T')\geq \pos(T'')$, and by symmetry $\pos(T')= \pos(T'')$.
 \end{proof}

\subsection{The symmetric case}\label{sec:sym}

Let $\kk$ be an Hilbert-Schmidt non-negative symmetric kernel. As $R_0[\kk]=0$ implies
$R_e[\kk]=0$ by \eqref{eq:r(A)r(B)}, we shall only consider the case $R_0[\kk]>0$. We now
prove Theorem~\ref{th:hill-longini} when $\kk$ is symmetric with finite double norm in
$L^2(\mu)$ and $R_0[\kk]>0$.

\subsubsection{The convex case}\label{sec:convexe}

The proof relies on an idea from \cite{Fried80} (see therein just before Theorem 4.3). Let
$\kk$ be an Hilbert-Schmidt non-negative symmetric kernel such that $ \spec(T_\kk)\subset
\R_+$, where $ T_\kk$ is the corresponding integral operator on $L^2(\mu)$. Since $T_\kk$
is a self-adjoint positive semi-definite operator on $L^2(\mu)$, there exists a
self-adjoint positive semi-definite operator $Q$ on $L^2(\mu)$ such that $Q^2=T$. Recall
that for a real-valued function $u$ defined on $\Omega$, $M_u$ denotes the multiplication
by $u$ operator. Thanks to \eqref{eq:r(AB)}, we have for $\eta\in \Delta$:
\[
  R_e[\kk](\eta)=\rho(T_\kk \,M_\eta)=\rho(Q^2\,
  M_\eta)= \rho ( Q \,M_\eta\, Q). 
\]
Since the self-adjoint operator $Q\, M_{\eta}\, Q$ (on $L^2(\mu)$) is also positive
semi-definite, we deduce from the Courant-Fischer-Weyl min-max principle that:
\[
  R_e[\kk](\eta)= \rho\left(Q\, M_{\eta}\, Q\right)=\sup_{u\in L^2(\mu)\setminus\set{0}}
  \frac{\langle u, Q \, M_{\eta} \,Qu \rangle}{ \langle u,u \rangle}\cdot
\]
Since the map $\eta \mapsto \langle u, Q \, M_{\eta} \,Q u \rangle$ defined on $\Delta$
is linear, we deduce that $\eta\mapsto R_e[\kk](\eta)$ is convex as a supremum of linear
functions.

\subsubsection{The concave case}\label{sec:concave}

Let $\kk$ be an Hilbert-Schmidt non-negative symmetric kernel such that $\pos(T_\kk)=1$.
In particular $\kk$ is monatomic, see Lemma~\ref{lem:R0simple}. Let $\Delta^*$ be the
subset of $\Delta$ of the functions which are bounded away from 0. The set $\Delta^*$ is a
dense convex subset of $\Delta$. So its suffice to prove that $R_e=R_e[\kk]$ is concave on
$\Delta^*$. Let $\eta_0$, $\eta_1$ be elements of $\Delta^*$, and set $\eta_\alpha =
(1-\alpha)\eta_0 + \alpha\eta_1$ for $\alpha\in [0, 1]$ (which is also an element of
$\Delta^*$). We write $T_\alpha = T_{\kk\eta_\alpha}$, so that $T_\alpha = T_0 + \alpha T_\kk M$, where $M$ is the multiplication by $(\eta_1 -
\eta_0)$ operator, and:
\[
  R(\alpha) = R_e(\eta_\alpha)=\rho(T_\alpha)=\rho(T_0 + \alpha T_\kk
  M). 
\]
So, to prove that $R_e$ is concave on $\Delta^*$ (and thus on $\Delta)$, it is enough to
prove that $\alpha\mapsto R(\alpha)$ is concave on $(0, 1)$. As $\eta_\alpha$ is also
bounded away from 0, we get that $\kk \eta_\alpha$ is monatomic and its spectral radius
$R(\alpha)$ is positive and a simple eigenvalue, thanks to Lemma~\ref{lem:R0simple}.
Thanks to Sylvester's inertia theorem, see Theorem~\ref{lem:sylvester} (with $f=1$ and
$g=\eta_\alpha$), we also get that $\pos(T_\alpha)=1$. \medskip

We consider the following scalar product on $L^2(\mu)$ defined by $\braket{u,v}_\alpha =
\braket{u,\eta_\alpha v}$. The operator $T_\alpha$ is self-adjoint and compact on
$L^2(\eta_\alpha \rd \mu)$ with spectrum $\spec(T_\alpha)$ thanks to
Lemma~\ref{lem:prop-spec-mult}~\ref{item:density-mult}. Let $(\lambda_n, n\in I=\lb 0, N
\lb)$, with $N\in \N \cup\{\infty \}$ be an enumeration of the non-zero eigenvalues of
$T_\alpha$ with their multiplicity so that $\lambda_0=R(\alpha)>0$ and thus $\lambda_n<0$
for $n\in I^*=I\setminus \{0\}$; and denote by $(u_n, n\in I)$ a corresponding sequence of
orthogonal eigenvectors. The functions $v_\alpha=u_0$ and $\phi_\alpha=\eta_\alpha u_0$
are the right and left-eigenvectors for $T_\alpha$ (seen as an operator on $L^2(\mu)$)
associated to $R(\alpha)$.

\medskip

We now follow \cite{EffectivePertuBenoit} to get that
$\alpha \mapsto R(\alpha)=\rho(T_0 + \alpha T_\kk M)$ is analytic and
compute its second derivative. Let $\pi_\alpha$ be the projection on
the ($\braket{\cdot,\cdot}_\alpha$)-orthogonal of $v_\alpha$, and
define:
\[
  S_\alpha = (T_\alpha - R(\alpha))^{-1} \pi_\alpha.
\]
In other words, $S_\alpha$ maps $u_0$ to $0$ and $u_i$ to
$(\lambda_i - R(\alpha))^{-1}\, u_i$.
Let $\alpha\in(0,1)$ and $\varepsilon$ small enough so that
$\alpha+\varepsilon\in [0, 1]$. We have:
\[
    T_{\alpha+\varepsilon} = T_\alpha + \varepsilon T_\kk M,
  \]
  and thus $\norm{T_{\alpha+\varepsilon}
      -T_\alpha}_{L^2(\eta_\alpha \rm{d} \mu)}=O( \varepsilon)$.
    Using \cite[Theorem~2.6]{EffectivePertuBenoit} on the Banach space
    $L^2(\eta_\alpha\, \mathrm{d}\mu)$, we get that:
   \[
      R(\alpha + \varepsilon)
      = R(\alpha) + \varepsilon\braket{v_\alpha,T_\kk M v_\alpha}_\alpha
  - \varepsilon^2\braket{v_\alpha, T_\kk M S_\alpha T_\kk M v_\alpha}_\alpha
  + O(\varepsilon^3).
\]
Let $N_\alpha=M_{1/\eta_\alpha}M=MM_{1/\eta_\alpha}$ be the
multiplication by $(\eta_1-\eta_0)/\eta_\alpha$ bounded operator. 
Since $\alpha \mapsto R(\alpha)$ is analytic and $T_\kk$ self-adjoint
(with respect to $\langle \cdot, \cdot \rangle$), we get that:
\begin{align*}
    R''(\alpha)
    &= - 2 \braket{v_\alpha,T_\kk M S_\alpha T_\kk M v_\alpha}_\alpha \\
    &= - 2 \braket{M T_\alpha v_\alpha, S_\alpha T_\kk M v_\alpha} \\
    &= - 2 R(\alpha) \braket{ M v_\alpha, S_\alpha T_\kk M v_\alpha}\\
    &= - 2 R(\alpha) \braket{ N_\alpha v_\alpha, S_\alpha T_\alpha N_\alpha v_\alpha}_\alpha.
  \end{align*}
  Since the kernel and the image of $T_\alpha$ are orthogonal (in
  $L^2(\eta_\alpha \rd \mu)$), and the latter is generated by
  $(u_n, n\in I)$, we have the decomposition
  $N_\alpha v_\alpha = g+ \sum_{n\in I} a_nu_n$ with
  $g\in \mathrm{Ker}(T_\alpha)$ and
  $a_n=\langle N_\alpha v_\alpha, u_n \rangle_\alpha$. This gives, with
  $I^*=I\setminus \{0\}$:
\begin{equation}\label{eq:R2}
    R''(\alpha) = 2 R(\alpha) \sum_{n\in I^*} \frac{\lambda_n}{R(\alpha) - \lambda_n} \, a_n^2\, \braket{u_n,u_n}_\alpha. 
\end{equation}
Since $\lambda_n<0$ for all $n\in I^*$, we deduce that $R''(\alpha)\leq
0$ and thus $\alpha \mapsto R(\alpha)$ is 
concave on $[0, 1]$. This implies that $R_e[\kk]$ is concave. 

\begin{remark}\label{rem:cvxe2}
  The same proof with obvious changes gives that if $\kk$ is an
  Hilbert-Schmidt non-negative symmetric monatomic (and thus
  quasi-irreducible) kernel such that $\nega(T_\kk)=0$, then $R_e[\kk]$
  is convex on $\Delta$. This result is however less general than the
  one obtained in Section~\ref{sec:convexe}. 
\end{remark}

\subsection{Proof of Theorem \ref{th:hill-longini} and Proposition \ref{prop:hill-longini}}
\label{sec:proof-HL}

We first consider the following technical Lemma.
\begin{lemma}
  \label{lem:p>p}
   Let $\kk$ be a
 diagonally HS kernel, with the HS kernel $\kk'$ from
 \eqref{eq:rep:k=fkg}. We have:
 \[
   \pos(T_\kk)\leq \pos(T_{\kk'})
   \quad\text{and}\quad
   \nega(T_\kk)\leq \nega(T_{\kk'}).
 \]
 If furthermore $\kk$ is strongly diagonally HS, then the previous
 inequalities are in fact equalities. 
\end{lemma}

\begin{proof}
  We only consider the number of positive eigenvalues as the number of negative
  eigenvalues can be handled similarly. Let $f,g$ be the functions from
  \eqref{eq:rep:k=fkg} and $v_n$ defined in \eqref{eq:def-vn} for $n\in \N^*$. For
  simplicity, we write $\pos(\kk'')$ for $\pos(T_{\kk''})$ when $\kk''$ is a kernel with
  finite double norm. Let $m\in \N^*$. As the function $w_{n,m}=\sqrt{fg}\, v_n + m^{-1}
  (1-v_n)$ is bounded and bounded away from 0, we deduce from the Sylvester's inertia
  Theorem~\ref{lem:sylvester} that:
  \begin{equation}
    \label{eq:pk-nm}
    \pos(\kk')=\pos\left(w_{n,m}\, \kk'\,w_{n,m} \right). 
  \end{equation}
  Notice that $\lim_{m\rightarrow \infty } \norm{ \sqrt{fg}\, v_n\, \kk' \,v_n \sqrt{fg}
  \, - w_{n,m}\, \kk'\,w_{n,m}}_{2,2}=0$. Letting $m$ goes to infinity, we deduce from
  \eqref{eq:cv-mult-vp} in Corollary~\ref{cor:cv-kn} and the fact that the spectrum is
  real that:
  \begin{equation}
    \label{eq:p>p-1}
    \pos(\kk')\geq \pos\left(\sqrt{fg}\, v_n\, \kk' \,v_n \sqrt{fg} \right).
  \end{equation}
  We also deduce from Remark~\ref{rem:hk/h-mult}, with $h=\sqrt{f/g}\, v_n + (1-v_n)$
  that:
  \[
    \pos\left(\sqrt{fg}\, v_n \, \kk'\, v_n \sqrt{fg} \right)
    =\pos\left(f v_n\, \kk'\, v_n g \right).
  \]
  Recall $\kk$ has a finite double norm in some $L^p$. By monotone convergence, we get
  that $\lim_{m\rightarrow \infty } \norm{ f\, \kk'\, g - f\, v_n\, \kk' \,v_n g
  }_{p,q}=0$. Letting $n$ goes to infinity, we also deduce from \eqref{eq:cv-mult-vp} in
  Corollary~\ref{cor:cv-kn} and the fact that the spectra of $T_{fv_n\,\kk' \, v_n g}$ and
  $T_{f\, \kk'\, g}$ are real according to Lemma~\ref{lem:Tk-real}, that:
  \begin{equation}
    \label{eq:p>p-2} \liminf_{n\rightarrow \infty }p\left(f v_n\, \kk'\, v_n g \right)
    \geq p\left(f \, \kk'\,g \right). 
  \end{equation}
  Thus, we have $\pos(\kk')\geq \pos(\kk)$. \medskip

  Notice that if $\kk$ is strongly diagonally HS, then $v_n=1$ for $n$ large enough, so
  that inequalities \eqref{eq:p>p-1} and \eqref{eq:p>p-2} are in fact equalities and thus
  $\pos(\kk')=\pos(\kk)$. 
\end{proof}

\begin{proof}[Proof of Proposition \ref{prop:hill-longini}]
  We only prove \ref{item:prop-hill-longini-cave} as the proof of
  \ref{item:prop-hill-longini-cvx} is similar and easier for the last part. We keep
  notations from the proof of Lemma~\ref{lem:p>p}. Assume that $\pos(\kk')=1$. We deduce
  from \eqref{eq:pk-nm} and from Section~\ref{sec:concave} that $R_e[w_{n,m} \, \kk'\,
  w_{n,m} ]$ is concave. We deduce from Corollary~\ref{cor:cv-kn}, letting $m$ goes to
  infinity, that $R_e[\sqrt{fg}\, v_n \, \kk'\, v_n \sqrt{fg}]$ is concave. Use
  Lemma~\ref{lem:hk/h}~\ref{lem:k=hk/h} with $h=\sqrt{f/g}\, v_n + (1-v_n)$ to obtain that
  $R_e[f v_n\, \kk'\, v_n g]$ is concave. Then, letting $n$ goes to infinity and using
  again Corollary~\ref{cor:cv-kn}, we deduce that $R_e[f\kk' g]=R_e[\kk]$ is concave. 
  \medskip

  Use also Lemma~\ref{lem:p>p} to get $\pos(\kk)\leq \pos(\kk')$. Now if $\pos(\kk)=0$,
  then we have that $R_0[\kk]=0$ which is equivalent to $R_0[\ind{\{\kk>0\}}]=0$. Since
  $\{\kk>0\}=\{\kk'>0\}$, this is also equivalent to $R_0[\kk']=0$. As this is ruled out
  because $\pos(\kk')=1$, we deduce that $\pos(\kk)=1$. 
\end{proof} 

\begin{proof}[Proof of Theorem \ref{th:hill-longini}]
  The result is an immediate consequence of Proposition
  \ref{prop:hill-longini} and the second part of Lemma~\ref{lem:p>p}.
\end{proof}

\section{Three properties of the Pareto and anti-Pareto frontiers}
\label{sec:3-P-et-AP}

We introduce in Section~\ref{sec:P-et-AP} the bi-objective minimization problem, where
one tries to minimize simultaneously the cost of the vaccination and the effective 
reproduction number, and recall results from~\cite{ddz-theo} on the Pareto and
anti-Pareto optimal strategies and frontiers. Then, we derive in Section~\ref{sec:ray}
the existence of Pareto optimal rays as soon as there exists a Pareto optimal
strategy uniformly strictly bounded from above by $1$. We prove in
Section~\ref{sec:cordon} that creating a cordon sanitaire is not the worst idea in the
sense that it is not anti-Pareto optimal (and it can be Pareto optimal or not). 
Eventually, in Section~\ref{sec:independance} we give a characterization of
$\cmir=\Cinf(0)$ using the notion of independent set from graph theory.

\subsection{Pareto and anti-Pareto frontiers}\label{sec:P-et-AP}

We quantify the cost of the vaccination strategy $\eta\in \Delta$ by a
function $C: \Delta \rightarrow \R^+$, and we assume that $C(\un)=0$
(doing nothing costs nothing), $C$ is non-increasing (doing more costs
more) and continuous for the weak topology
 on $\Delta$ defined in
Section~\ref{sec:weak-topo}. Recall that $1-\eta$ represents the proportion of
the population which has been vaccinated when using the strategy
$\eta$. One natural choice is the uniform cost
function $C=\costu$
defined for $\eta\in \Delta$ by:
\begin{equation}
  \label{eq:def-C}
  \costu(\eta)=\int_\Omega (1-\eta)\, \rd \mu.
\end{equation}

In \cite{ddz-theo}, we formalized and study the problem of optimal allocation strategies
for a perfect vaccine. This question may be viewed as a bi-objective minimization
problem, where one tries to minimize simultaneously the cost of the vaccination and the
effective reproduction number: 
\begin{equation}
   \label{eq:bi-min}
  \min_{\Delta} (C,R_e).
\end{equation}
\medskip

We briefly summarize the results from \cite{ddz-theo}. We shall assume
that the kernel $\kk$ has a finite double norm, the loss function is
given by the effective reproduction function $R_e[\kk]$, and the cost
function $C$ is furthermore \emph{decreasing} (this is the case of the uniform cost), that
is, for any $\eta_1, \eta_2\in \Delta$:
\[
  \eta_1\leq \eta_2 \quad\text{and}\quad
  \int_\Omega \eta_1 \, \mathrm{d} \mu < \int_\Omega \eta_2 \, \mathrm{d} \mu
\,\Longrightarrow \, C(\eta_1)> C(\eta_2).
\]
\medskip

To be precise, the next results can be found in \cite[Propositions~5.4
and~5.5]{ddz-theo} (notice in particular, that Assumptions 4 and 5
holds thanks to Lemma~5.13 therein). 
By definition, we have $R_0=\max_\Delta \, R_e$ and we set
$ \cmax=\max_\Delta C$ which is positive as $C$ is decreasing. 
Related to the minimization problem~\eqref{eq:bi-min}, we shall
consider~$\mir$ the \emph{optimal loss} function and~$\Cinf$ the
\emph{optimal cost} function defined by:
\begin{align*}
\mir (c) &= \min \, \set{ \loss(\eta) \, \colon \, \eta \in \Delta,\,
  C(\eta) \leq c }\quad\text{for $c\in [0,\cmax]$},\\
  \Cinf(\ell) &= \min \, \set{ C(\eta) \, \colon \, \eta \in \Delta,\,
    \loss(\eta) \leq \ell } \quad \text{for $\ell \in [0,R_0]$}.
\end{align*}
We have $\Cinf(R_0)=0$ and $\mir(0)=R_0$ since $C$ is decreasing.
For convenience, we write $\cmir$ for the minimal cost required
to completely stop the transmission of the disease:
\begin{equation}
  \label{eq:def-cmir}
 \cmir=\Cinf(0) = \inf \{ c \in [0,\cmax] \, \colon \, \mir(c) = 0 \}.
\end{equation}
The function $\mir$ is continuous, decreasing on $[0, \cmir]$ and zero on $[\cmir,
1]$; the function $\Cinf$ is continuous and decreasing on $[0, R_0]$;
and the functions $\mir$ and $\Cinf$ are the inverse of each other, that is,
$\mir \circ \Cinf(\ell)=\ell$ for $\ell\in [0, R_0]$ and 
$\Cinf \circ \mir( c)=c$ for $c\in [0, \cmir]$. 
\medskip

We define the Pareto optimal strategies $\cp$ as the ``best'' solutions of the
minimization problem~\eqref{eq:bi-min} (we refer to \cite{ddz-theo} for a precise
justification of this terminology):
\[
  \cp=\left\{\eta\in \Delta\, \colon\, C(\eta)=\Cinf(R_e(\eta))
    \quad\text{and}\quad R_e(\eta)=\mir(C(\eta))
  \right\},
\]
and the Pareto frontier as their outcomes:
\[
  \F=\left\{(C(\eta), R_e(\eta))\, \colon\, \eta\in \cp \right\}.
\]
The set $\cp$ is a non empty compact (for the weak topology) in $\Delta$ and furthermore
the Pareto frontier can be easily represented using the graph of the optimal loss function
or cost function:
\[
  \F
  = \{(\Cinf(\ell), \ell) \, \colon \, \ell \in [0,R_0]\}
  = \{(c, \mir(c)) \, \colon \, c \in [0,\cmir]\}.
\]
\medskip

It is also of interest to consider the ``worst'' strategies
which can be viewed as solutions to the bi-objective maximization
problem: 
\begin{equation}
   \label{eq:bi-max}
  \max_{\Delta} (C,R_e). 
\end{equation}
To be precise, the next results can be found in \cite[Propositions~5.8
and~5.9]{ddz-theo} (notice in particular that Assumption 6 holds in
general but that Assumption 7 holds under the stronger condition
that the kernel $\kk$ is monatomic, see Section~5.4.2 therein).
Related to the maximization problem~\eqref{eq:bi-max}, we shall
consider~$\mar$ the \emph{optimal loss} function and~$\Csup$ the
\emph{optimal cost} function defined by:
\begin{align*}
\mar (c) &= \max \, \set{ \loss(\eta) \, \colon \, \eta \in \Delta,\,
  C(\eta) \geq c }\quad\text{for $c\in [0,\cmax]$},\\
  \Csup(\ell) &= \max \, \set{ C(\eta) \, \colon \, \eta \in \Delta,\,
    \loss(\eta) \geq \ell } \quad \text{for $\ell \in [0,R_0]$}.
\end{align*}
We have $\Csup(0)=\cmax$ and $\mar(\cmax)=0$ since $C$ is decreasing and
$C(\zero)=\cmax$. Since, for $\varepsilon\in (0, 1)$ we have
$C(\varepsilon \un )<\cmax$ as $C$ is decreasing and
$R_e(\varepsilon \un)=\varepsilon R_0>0$, we deduce that
$\Csup(0+)=\cmax$. For convenience, we write $\cmar$ for the maximal
cost of totally inefficient strategies:
\begin{equation}
    \label{eq:def-cmar}
  \cmar = \Csup(R_0)=\max \{ c \in [0,\cmax] \, \colon \, \mar(c) = R_0 \}.
\end{equation}
The function $\Csup$ is decreasing on $[0, R_0]$;
the function $\mar$ is constant equal to $R_0$ on $[0, \cmar]$; we have 
$\mar \circ \Csup(\ell)=\ell$ for $\ell\in [0, R_0]$.
  This latter
property implies that the function $\mar$ is continuous. 

We define the anti-Pareto optimal strategies $\cpa$ as the ``worst''
strategies, that is solutions of
the maximization problem~\eqref{eq:bi-max}:
\[
  \cpa=\left\{\eta\in \Delta\, \colon\, C(\eta)=\Csup(R_e(\eta))
    \quad\text{and}\quad R_e(\eta)=\mar(C(\eta)) \right\},
\]
and the anti-Pareto frontier as their outcomes:
\[
  \AF=\left\{(C(\eta), R_e(\eta))\, \colon\, \eta\in \cpa \right\}.
\]
The set $\cp$ is non empty 
and furthermore the Pareto frontier can be easily represented using the
graph of the optimal cost function:
\begin{equation}
   \label{eq:FL=L*}
  \AF
   = \{(\Csup(\ell), \ell) \, \colon \, \ell \in [0,R_0]\}.
\end{equation}
We also have that the feasible region or set of possible outcomes for $(C, R_e)$:
\[
  \FF=\left\{(C(\eta), R_e(\eta))\, \colon \, \eta\in \Delta \right\}
\]
is compact,
  path connected, and its complement is connected in $\R^2$. It is the whole region
  between the graphs of the one-dimensional value functions:
  \begin{align*}
    \FF &= \{ (c,\ell) \in [0, \cmax]\times [0, R_0]\,\colon\, 
    \mir(c) \leq \ell \leq \mar(c) \} \\
	&= \{ (c,\ell) \in [0, \cmax]\times [0, R_0] \,\colon\, 
	\Cinf(\ell) \leq c \leq \Csup(\ell)\}.
  \end{align*}
\medskip

If furthermore $\kk$ is monatomic with atom $\oa$, then thanks to
\cite[Lemma~5.13]{ddz-theo}, we have $ \cmar=C(\ind{\oa} )$ (which is 0
if $\kk$ is irreducible); the function $\mar$ is continuous, decreasing
on $[ \cmar, \cmax]$; the function $\Csup$ is continuous and decreasing
on $[0, R_0]$; the functions $\mar$ and $\Csup$ are the inverse of each
other, that is, $\mar \circ \Csup(\ell)=\ell$ for $\ell\in [0, R_0]$ and
$\Csup \circ \mar( c)=c$ for $c\in [ \cmar, \cmax]$; and the set $\cpa$
is compact and
$ \AF = \{(c, \mar(c)) \, \colon \, c \in [\cmar, \cmax]\}$. \medskip

\begin{figure}[t]
  \begin{subfigure}[T]{.5\textwidth}
    \centering
    \includegraphics[page=1]{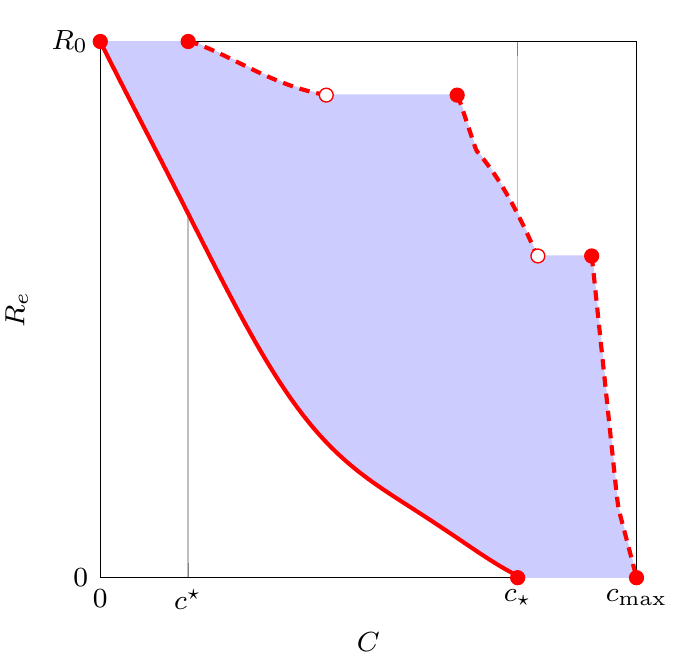}
    \caption{General kernel.}
    \label{fig:generic_frontiers-gen}
  \end{subfigure}%
  \begin{subfigure}[T]{.5\textwidth}
    \centering
    \includegraphics[page=2]{generic-frontier}
    \caption{Monatomic kernel.}
  \end{subfigure}

  \begin{subfigure}[T]{.5\textwidth}
    \centering
    \includegraphics[page=3]{generic-frontier}
    \caption{Irreducible kernel.}
    \label{fig:generic_frontiers-irr}
  \end{subfigure}%
  \begin{subfigure}[T]{.5\textwidth}
    \centering
    \includegraphics[page=4]{generic-frontier}
    \caption{Kernel strictly positive almost surely.}
  \end{subfigure}
  \caption{Generic aspect of the feasible region (light blue), the
    Pareto frontier (thick red line) and the anti Pareto frontiers
    (dashed red line) for the cost function~$R_e[\kk]$, with kernel
    $\kk$, and a continuous decreasing cost function $C$.
    }
    \label{fig:generic_frontiers}
\end{figure}

We plotted in Figure~\ref{fig:generic_frontiers} the typical Pareto and anti-Pareto
frontiers for a general kernel (notice the anti-Pareto frontier is not connected \textit{a
priori}), a monatomic kernel (notice the anti-Pareto frontier is connected), and a
positive kernel. In the latter case, the properties of the frontiers are stated in 
the next lemma.

\begin{lem}\label{lem:k>0-c}
  Suppose that the cost function $C$ is continuous decreasing with $C(\un)=0$ and
  consider the loss function $R_e[\kk]$, with $\kk$ a finite double norm kernel such 
  that a.s. $\kk>0$. Then, we have $R_0[\kk]>0$, $\cmar=0$, $\cmir=\cmax$ and the
  strategy $\un$ (resp. $\zero$) is the only Pareto optimal as well as the only 
  anti-Pareto optimal strategy with cost $c=0$ (resp. $c=1$).
\end{lem}

\begin{proof}
  Since $\kk>0$, we get that $\kk$ is irreducible (and thus monatomic) and $R_0>0$, thanks to
  Lemma~\ref{lem:R0simple}. We get that $\cmar=0$. This implies
  that the strategy $\un$ is anti-Pareto optimal. As $C$ is decreasing,
  we also get that the strategy $\un $ is Pareto optimal.

  Let $\eta\in \Delta$ be different from $\zero$. We get that
  the kernel $\kk \eta$ restricted to the set of positive $\mu$-measure
  $\{\eta>0\}$ is positive, thus
  $\kk \eta$ is monatomic (with $\oa=\{\eta>0\}$ and
  $\oi=\oa^c$). Thanks Lemma~\ref{lem:R0simple}, we get
  that $R_e(\eta)>0$. This readily implies that $\cmir=\cmax$ and that
  the strategy $\zero$ is Pareto optimal. As $C$ is decreasing, we also
  get that the strategy $\zero $ is anti-Pareto optimal.
\end{proof}

\subsection{Optimal ray}\label{sec:ray}

As the loss function $R_e$ is convex and homogeneous, and if the cost function is affine,
then the set $\mathcal{P}$ of Pareto optimal strategies may contains a non-trivial optimal
ray $\{\lambda \eta\, \colon\, \lambda\in [0, 1]\}$. This optimal ray has already been
observed in finite dimension, see \cite{poghotanyan_constrained_2018}.

\begin{proposition}[Optimal ray]\label{prop:critical-ray}
  Suppose that the cost function $C$ is continuous decreasing and affine and that the loss
  function $R_e[\kk]$, with $\kk$ a finite double norm kernel, is convex. If there exists
  a Pareto optimal strategy $\eta_\star \in \mathcal{P}$ such that $\sup_\Omega \eta_\star
  \in (0, 1)$, then the strategies $\lambda \eta_\star$ are Pareto optimal for all
  $\lambda \in [0, 1/ \sup_\Omega \eta_\star]$.
\end{proposition} 

\begin{remark}\label{rem:cric-ray}
  Suppose assumptions of Proposition~\ref{prop:critical-ray} hold so that there
  is an optimal ray $\{ \lambda \eta_\star \, \colon \, \lambda \in [0,1] \} \subset
  \mathcal{P}$, where $\sup_\Omega \eta_\star = 1$. Then, by homogeneity of the loss
  function, the Pareto frontier has a linear part (from $(C(\eta_\star),
  \loss( \eta _\star))$ to $(\cmax,0)$). 
\end{remark}

\begin{remark}
  Suppose that $C$ is continuous decreasing and affine and that $R_e$ is
  concave. With a similar proof (but for the last part which has to be
  replaced by the fact that $\Csup(0+)=\cmax$), we can show that if
  $\eta^\star$ is anti-Pareto optimal such that
  $\sup_\Omega \, \eta^\star \in (0,1)$, then $\lambda \eta^\star$ is
  also anti-Pareto optimal for all
  $\lambda \in [0,1/ \sup_\Omega \eta^\star]$.
\end{remark}

\begin{proof}[Proof of Proposition~\ref{prop:critical-ray}]
  Assume there exists $\eta_\star \in \mathcal{P}$ such that 
  $\sup_\Omega \eta_\star \in (0, 1)$. Let
  $\lambda \in (0, 1 / \sup_\Omega \eta_\star]$, so that $\lambda \eta_\star\in
  \Delta$, and let $\eta \in \Delta$ such that
  $\loss(\eta) \leq \loss(\lambda \eta_\star)$, and thus $\loss(\eta)
  \leq \lambda \loss(\eta_\star)$. Since $\sup_\Omega \eta_\star < 1$,
  there exists $s \in (0,1]$ such that
  $(1-s) \eta_\star + s \eta / \lambda \in \Delta$. Using the
  homogeneity and the convexity of $\loss$, we get:
  \begin{align*}
    \loss( (1-s) \eta_ \star + s \eta/ \lambda)
  &= \frac{1}{\lambda} \loss((1-s) \lambda \eta_\star + s \eta ) \\
    &\leq (1 - s) \loss(\lambda \eta_\star) / \lambda + s \loss(\eta)/\lambda \\
    &\leq \loss(\eta_\star). 
  \end{align*}
  Since $\eta_\star$ is Pareto optimal, we deduce that
  $C((1-s) \eta_\star + s \eta / \lambda) \geq C(\eta_\star)$. Since $C$
  is affine, we get that $C(\eta) \geq C(\lambda \eta_\star)$. Hence,
  $\lambda \eta_\star$ is solution of the problem $\min C(\eta)$ for
  $\eta\in \Delta$ such that $\loss(\eta)\leq \ell$ with
  $\ell = \loss(\lambda \eta_\star)$. We conclude that
  $\lambda \eta_\star$ is Pareto optimal using
  \cite[Proposition~5.5~(ii)]{ddz-theo}. Use that the Pareto optimal set
  is closed, see \cite[Corollary~5.7]{ddz-theo} to get that $\lambda
  \eta_\star$ is Pareto optimal for $\lambda=0$. 
\end{proof}

\subsection{Creating a cordon sanitaire is not the worst idea}\label{sec:cordon}

We say a strategy $\eta\in \Delta$ is a \emph{cordon sanitaire} or \emph{disconnecting}
(for the kernel $\kk$) if $\eta\neq \zero$ and the kernel $\kk$ restricted to the set
$\{\eta>0\}$ is not connected (or equivalently not irreducible). We make some elementary
comments on disconnecting strategies. 

\begin{remark}\label{rem:disco}
  Let $\kk$ be a kernel on $\Omega$. 
  \begin{propenum}
  \item The strategy $\eta=\un$ is disconnecting if and only if $\kk$ is not connected.
  \item A strategy $\eta$ is disconnecting if and only if the strategy $\ind{\{\eta>0\}}$
    is disconnecting.
  \item\label{item:disco-supp} If $\kk>0$, then there is no disconnecting strategy.
  \end{propenum}
\end{remark}

The next proposition states that if the strategy $\eta$ is anti-Pareto optimal for a
kernel $\kk$ and non zero, then the kernel $\kk$ restricted to $\{\eta>0\}$ is irreducible
and thus the kernel $\ind{\{\eta>0\}} \kk\ind{\{\eta>0\}}$ is quasi-irreducible. Let us
remark that in general none of those implications are equivalences.

\begin{proposition}[A cordon sanitaire is never the worst idea]\label{prop:cut}
  Suppose that the cost function $C$ is continuous decreasing and consider the loss
  function $R_e[\kk]$, with $\kk$ a finite double norm kernel on $\Omega$ such that
  $R_0[\kk]>0$. Then, a disconnecting strategy is not anti-Pareto optimal.
\end{proposition}

In the non-oriented cycle graph from Example~\ref{ex:cycle-graph}, this property is illustrated in
Figure~\ref{fig:perf} as the disconnecting strategy ``one in $4$'', see
Figure~\ref{fig:cycle-kern-graph-disc}, is not anti-Pareto.

The proof of the proposition relies on the next lemma which is a direct application of
\cite[Lemma~11]{schwartz61} to our setting. For $A\in \cf$, let 
$\mult(\lambda, \kk, A)$ be the
multiplicity (possibly equal to 0)
of the eigenvalue $\lambda\in \C^*$ for the integral operator $T_{\kk
\ind{A}}$ associated to the
kernel $\kk\ind{A}$. 
\begin{lemma}
  \label{lem:ext-61}
  Let $\kk$ be kernel with finite double
  norm. Let $A, B\in \cf$ be such that $A\cap
  B=\emptyset$ a.s.\ and $\kk(B,A)=0$. For all
  $\lambda\in \C^*$, we have:
  \[
    \mult (\lambda, \kk, A\cup B)= \mult(\lambda, \kk, A) + \mult(\lambda, \kk, B), 
  \]
  and thus
  \begin{equation}
    \label{eq:R0=maxR0}
    R_e[\kk](\ind{A}+ \ind{B})=\max \big( R_e[\kk](\ind{A}),R_e[\kk](\ind{B})
    \big).
  \end{equation}
\end{lemma}
We are now in a position to prove Proposition~\ref{prop:cut}.

\begin{proof}[Proof of Proposition~\ref{prop:cut}]
  Let $\eta$ be a disconnecting strategy, and thus $\eta\neq \zero$. Since $\eta$ is
  disconnecting, that is, $\kk$ restricted to $\{\eta>0\}$ is not irreducible, we deduce
  there exists $A, B\in \cf$ such that $\mu(A)>0$, $\mu(B)> 0$, $(\kk \eta) (B, A)=0$ and
  a.s. $A\cup B=\{\eta>0\}$ and $A\cap B=\emptyset$. In particular \eqref{eq:R0=maxR0}
  holds with $\kk$ replaced by $\kk \eta$. First assume that $R_e[\kk\eta](\ind{A})\geq
  R_e[\kk\eta](\ind{B})$, so that \eqref{eq:R0=maxR0} yields: \[ R_e[\kk](\eta)=
  R_e[\kk\eta](\ind{A}+\ind{B})= R_e[\kk\eta](\ind{A}). \] For $\theta\in [0, 1]$, define
  the strategy $\eta_\theta=\eta \ind{A}+ \theta \eta \ind{B}$. We deduce that:
  \begin{align*}
    R_e[\kk](\eta_\theta)
    = R_e[\kk\eta_\theta](\ind{A} + \ind{B})
    &= \max (R_e[\kk\eta_\theta](\ind{A}), R_e[\kk\eta_\theta](\ind{B})) \\
    &=\max (R_e[\kk\eta](\ind{A}), \theta R_e[\kk\eta](\ind{B}))\\
    & = R_e[\kk\eta](\ind{A})\\
    &= R_e[\kk](\eta),
  \end{align*} 
  where we used \eqref{eq:R0=maxR0} with $\kk$ replaced by $\kk \eta_\theta$ for the
  second equality as $(\kk \eta_\theta) (B, A)=0$, and the homogeneity of the spectral
  radius in the third. Thus, the map $\theta\mapsto R_e[\kk](\eta_\theta)$ is constant on
  $[0, 1]$. Since $\mu(B)>0$ and $C$ is decreasing, we get that $\theta \mapsto
  C(\eta_\theta)$ is decreasing. This implies that $\eta_\theta$ is worse than $\eta$ for
  any $\theta\in [0, 1)$, and thus $\eta$ is not anti-Pareto optimal.

  The case $R_e[\kk\eta](\ind{B})\geq R_e[\kk\eta](\ind{A})$ is handled similarly.
\end{proof}

\begin{rem}\label{rem:plateau}
  If the kernel $\kk$ is irreducible, then the upper boundary of the set of outcomes $\FF$
  is the anti-Pareto frontier, see Figure~\ref{fig:generic_frontiers-irr} for an instance.
  We deduce from Proposition~\ref{prop:cut} that if $\eta_0$ is a disconnecting strategy,
  then we have $R_e[\kk](\eta_0)< \sup\{R_e[\kk](\eta)\, \colon\, C(\eta)=C(\eta_0)\}$.

  However, if the kernel $\kk$ is not irreducible, then the trivial strategy $\un $ is
  disconnecting. Furthermore, the upper boundary of the set of outcomes $\FF$ is not
  reduced to the anti-Pareto frontier, see Figure~\ref{fig:generic_frontiers-gen} for
  instance. In fact, there exists disconnecting strategies that are not anti-Pareto
  optimal, but whose outcomes lie on the flat parts of the upper boundary of $\FF$. In
  particular, such strategies have the worst loss given their cost. However, it is not
  difficult to check that they do not disconnect further than the trivial strategy $\un$.
\end{rem}

\subsection{A characterization of \texorpdfstring{$\cmir=\Cinf(0)$}{C*(0)}
  when the support of \texorpdfstring{$\kk$}{k} is symmetric}\label{sec:independance}

We characterize the Pareto optimal strategies which minimize $R_e$ when the kernel $\kk$
has a symmetric support; and we get a very simple representation of $\Cinf(0)$ when the
cost is uniform $C=\costu$.

\medskip

Let us first recall a notion from graph theory. If $G=(V,E)$ is an non-oriented graph with
vertices set $V$ and edge set $E$, an \emph{independent} set of $G$ is a subset $A\subset
V$ of vertices which are pairwise not adjacent, that is, $i,j\in A$ implies $i j\not\in
E$. The \emph{independence number} of a graph $G$, denoted by $\alpha(G)$, is the maximum
of $\sharp A/ \sharp G$, over all the independent sets $A$ of $G$. Following
\cite{hladky_independent_2020}, we generalize this definition to kernels.

\begin{definition}[Independent sets for kernels]\label{def:indep-number}
  Let $\kk$ be a kernel on $\Omega$. A measurable set $A \in \cf$ is an independent set
  of $\kk$ if $\kk=0$ $\mu^{\otimes 2}$-a.s.\ on $A\times A$. The independence number
  $\alpha(\kk )$ of the kernel $\kk $ is:
  \[
    \alpha(\kk ) = \sup\{ \mu(A)\,\colon\, \text{$A$ is an independent set of
    $\kk $}\}.
  \]
\end{definition}

A compactness argument will show that the supremum defining $\alpha$ is reached. 

\begin{proposition}[Existence of a maximal independent set]\label{prop:max-indep-set}
  For any kernel $\kk $ on $\Omega$, there exists an independent set $A$ of
  $\kk $ that is \emph{maximal}, in the sense that $\mu(A) =
  \alpha(\kk )$. 
\end{proposition}

\begin{proof}
  First, notice that the independent sets and maximal independent sets of a kernel
  $\kk $ depends only on the support $\set{\kk >0}$ of $\kk $.
  Therefore, the maximal independent sets of the kernel $\kk $ and of the kernel
  $\mathds{1}_{\set{\kk >0}}$ are the same. In particular, we can assume without
  loss of generality that the kernel $\kk$ is bounded.

  Let $(A_n, \, n \in \N)$ be a sequence of independent sets for $\kk$ such that:
  \[
    \lim\limits_{n \to \infty} \mu(A_n) = \alpha(\kk ).
  \]
  Since $\Delta$ is sequentially compact for the weak topology, up to
  taking a sub-sequence, we may assume that the sequence
  $(\mathds{1}_{A_n}, \, n \in \N)$ converges weakly to some function
  $g \in \Delta$. Since $\kk$ is bounded, the integral operator $T_\kk$
  is well defined. We deduce that $T_\kk(\mathds{1}_{A_n})$ belongs to
  $\Delta$ and converges a.s.\ towards $T_\kk(g)$. This implies that
  $\mathds{1}_{A_n}T_\kk(\mathds{1}_{A_n})$ converges weakly towards
  $g T_\kk(g)$. We deduce that:
  \[
    \int_\Omega g T_\kk(g) \, \mathrm{d} \mu
    = \lim\limits_{n \to \infty}
    \int_\Omega \mathds{1}_{A_n} T_\kk(\mathds{1}_{A_n}) \, \mathrm{d} \mu
    =\lim\limits_{n \to \infty} \kk(A_n, A_n)
    =0.
  \]
  As $g\in \Delta$, this implies that $\set{g>0} $ is an independent set of $\kk$ and thus
  $\mu\left(g>0\right) \leq \alpha(\kk)$. Besides, since $(\mathds{1}_{A_n} , n \in \N)$
  converges weakly to $g$, we get:
  \[
    \int_\Omega g \, \mathrm{d}\mu = \lim\limits_{n \to \infty} \mu(A_n) = \alpha(\kk ).
  \]
  This implies that
  $\mu\left(g>0 \right) \geq \int_\Omega g \, \mathrm{d}\mu =
  \alpha(\kk)$. We deduce that $\mu\left(g>0 \right)= \alpha(\kk)$, and
  since $\set{g>0}$ is an independent set, it is also maximal.
\end{proof}

In the following result, we prove that maximal independent sets provide optimal Pareto
strategies for the loss function $R_e$ and the cost function $\costu$ given
by~\eqref{eq:def-C} corresponding to the cost $\cmir=\Cinf(0)$, see also
Remark~\ref{rem:alpha=C(A)} for a general cost function. This property is illustrated in
Figure~\ref{fig:perf} where the Pareto frontier of the non-oriented cycle graph from
Example~\ref{ex:cycle-graph}, with $N=12$, is plotted; it is possible to prevent
infections without vaccinating the whole population as $\cmir=1/2<1=\cmax$.

\begin{proposition}\label{prop:CRe(0)}
  Let $\kk$ be a finite double norm kernel on $\Omega$ such that its support,
  $\{\kk>0\}$, is a symmetric subset of $\Omega^2$ a.s. We consider the cost
  $C=\costu $ given by~\eqref{eq:def-C}. For any maximal independent set
  $A_\star$ of $\kk$, the strategy $\ind{A_\star}$ is Pareto optimal for the
  loss $R_e[\kk]$ and we have:
  \begin{equation}
   \cmir= \Cinf(0) = C( \ind{A_\star})= 1 -\alpha(\kk).
  \end{equation} 
\end{proposition}

\begin{remark}\label{rem:alpha=C(A)}
  Definition \ref{def:indep-number} on maximal independent set is in
  fact associated to the uniform cost $C=\costu$. More generally, we
  could define the independence number $\alpha_C(\kk )$ of the
  kernel $\kk$ with respect to a decreasing continuous cost
  function $C$ (recall the convention $C(\un)=0$ and $\cmax=C(\zero)$) as:
  \[
    \alpha_C(\kk) = \sup \{ \cmax - C(\ind{A})\,\colon\, \text{$A$
      is an independent set of $\kk $}\}.
  \]
  The notations are consistent as $\alpha_C=\alpha$ for $C=\costu$.
  Adapting the proof of Proposition \ref{prop:max-indep-set}, we get
  that for any kernel $\kk$ on $\Omega$, there exists an independent set
  $A$ of $\kk$ that is $C$-maximal, in the sense that
  $ \alpha_C(\kk)=\cmax - C(\ind{A})$. Following the proof of
  Proposition~\ref{prop:CRe(0)}, we then get that if the finite double
  norm kernel $\kk$ on $\Omega$ has its support, $\{\kk>0\}$ which is a
  symmetric subset of $\Omega^2$ a.s., then for any $C$-maximal independent
  set $A_\star$ of $\kk$, the strategy $\mathds{1}_{A_\star}$ is Pareto
  optimal for the loss $R_e[\kk]$ and the cost $C$. Furthermore, we
  have:
  \[
   \cmir= \Cinf(0) = C( \mathds{1}_{A_\star})= \min \{C(\ind{A})\,\colon\, \text{$A$
      is an independent set of $\kk $}\}.
  \]
\end{remark}

\begin{proof}[Proof of Proposition~\ref{prop:CRe(0)}]
  The existence of a maximum independent set $A$ is given by Proposition
  \ref{prop:max-indep-set}. The effective reproduction number obviously vanishes for the
  strategy $\mathds{1}_A$ with cost $1 - \alpha(\kk)$ as $(T_{\kk \mathds{1}_A})^2=T_\kk\,
  T_{\mathds{1}_A\kk \mathds{1}_A}=0$. Now, let $\eta \in
  \Delta$ be such that $R_e[\kk](\eta) = 0$. To complete the proof of the
  proposition, it is enough to prove that $\costu(\eta)\geq 1 -\alpha(\kk)$.

  Since $R_e[\kk](\eta) = 0$, the spectral radius of $T_{\kk\eta}$ is equal to $0$. Let
  $\varepsilon>0$ and consider the kernel $\kk_\varepsilon$ defined on $\Omega$ by:
  \[
    \kk_\varepsilon(x,y)= \mathds{1}_{\set{ \kk(x,y)>\varepsilon}}.
  \]
  Since $T_{\kk\eta}- \varepsilon T_{\kk_\varepsilon \eta}$ is a positive operator, we
  deduce from~\eqref{eq:r(A)r(B)} that $\varepsilon \rho(T_{\kk_\varepsilon \eta })= \rho
  (\varepsilon T_{\kk_\varepsilon\eta}) \leq \rho (T_{\kk\eta})=0$ and thus
  $\rho(T_{\kk_\varepsilon\eta})=0$. Set $\kk' = \mathds{1}_{\set{\kk>0}}$. Since
  $\lim_{\varepsilon\rightarrow 0+}\norm{\kk_\varepsilon - \kk'}_{p, q}=0$, we deduce from
  Proposition \ref{prop:Re-stab} on the stability of $R_e$ that $\rho(T_{\kk'
  \eta})=R_e[\kk'](\eta) = \lim_{\varepsilon\rightarrow 0+} R_e[\kk_\varepsilon](\eta) =
  \lim_{\varepsilon\rightarrow 0+} \rho(T_{\kk_\varepsilon \eta})=0$. As the support of
  $\kk$ is symmetric, we deduce that the kernel $\kk'$ is symmetric. According
  to~\eqref{eq:r(AB)}, we have:
  \[
    \rho(T_{\kk''})= \rho(T_{\kk'\eta})=0,
  \]
  with $\kk''= \sqrt{\eta}\, \kk'\, \sqrt{\eta}=\sqrt{\eta}\, \mathds{1}_{\set{\kk>0}}\, 
  \sqrt{\eta}$. Since the kernel $\kk''$ is symmetric, non-negative and bounded by $1$,
  this implies that $\kk''=0$ $\mathrm{d}\mu^{\otimes 2}$-a.s., and thus $\set{\eta>0}$ is
  an independent set for $\kk$. This gives $\mu\left(\eta>0\right) \leq \alpha(\kk)$.
  Therefore, we have the following lower bound for the cost $\costu(\eta)$:
  \[
    \costu(\eta) = 1 - \int_\Omega \eta \, \mathrm{d} \mu 
    \geq 1 - \mu\left(\eta>0\right) \geq 1 - \alpha(\kk). 
  \]
  This ends the proof of the proposition.
\end{proof}

\section{Pareto and anti-Pareto frontiers for reducible kernels}\label{sec:reducible}

When the kernel $\kk$ is ``truly reducible'' (corresponding to the set of indices $I$
below to be such that $\sharp I\geq 2$), it is natural to ask whether the Pareto and
anti-Pareto frontiers of the subsystems entirely characterize the frontiers for $\kk$, and
in what sense the optimization problems can be ``reduced'' to the separate study of each
irreducible component.

We can achieve an elementary description of the anti-Pareto frontier when the kernel is
not reducible using a Frobenius decomposition, see \cite{victory82, victory93} and
\cite{schwartz61} or the ``super diagonal'' form, see \cite[Part~II.2]{dowson}. For
convenience, we follow \cite{schwartz61}, see also \cite[Lemma~5.17]{bjr} in the case
$\kk$ symmetric.

Let $\kk$ be a kernel on $\Omega$ with finite double norm. Let $\ca$ be
the set of $\kk$-invariant sets, and notice that $\ca$ is stable by
countable unions and countable intersections. Let $\sigma(\ca)$ be the
$\sigma$-field generated by $\ca$, and we denote by $(\Omega_i, i\in I)$
the at most countable (but possibly empty) collection of atoms with
respect to the measure $\mu$. Notice that the atoms are define up to an
a.s.\ equivalence and can be chosen to be pair-wise disjoint. For
$i\in I$, we set:
\begin{equation}\label{eq:ki}
  \kk_i=\ind{\Omega_i} \kk \ind{\Omega_i},
\end{equation}
which is a kernel on $\Omega$ with finite double norm. Set
$\Omega_0=\left(\cup_{i\in I} \Omega_i\right)^c$ (and assume the set of
indices $I$ has been chosen so that it does not contain 0). Thanks to
\cite[Lemma~12]{schwartz61} or \cite[Section~II]{victory82}, there
exists a total order, say $\preccurlyeq$, on $I$ (not 
unique in general) such that for all $i, j\in I$:
\begin{propenum}
\item $j\prec i $ implies $\kk(\Omega_i, \Omega_j)=0$. In the
  epidemiology setting, $ j \prec i$ means that the sub-population $\Omega_j$ can not
  infect the sub-population $\Omega_i$. 

\item $\mu(\Omega_i)>0$ and $\kk$ restricted to $\Omega_i$ is irreducible and has positive
  spectral radius, that is
  $\kk_i$ is quasi-irreducible, and
  $R_e[\kk](\ind{\Omega_i})=R_0[\kk_i]>0$.
\item $\kk$ reduced to $\Omega_0$ is quasi-nilpotent, that is
  $R_e[\kk](\ind{\Omega_0})=0$.
\item For all $\lambda\in \C^*$:
 \begin{equation}
   \label{eq:mult-vpFrob}
   \mult(\lambda, \kk)= \sum_{i\in I} \mult(\lambda, \kk_i).
  \end{equation} 
\end{propenum}
The next remark gives some elementary results related to the
Frobenius decomposition. 

\begin{remark}\label{rem:Frob}
  Recall $R_0[\kk]$ denote the spectral radius of the integral operator with
  kernel $\kk$ and that $\{\kk\equiv 0\} = \{x\in \Omega\, \colon\, \kk(x, \Omega)+
  \kk(\Omega, x)=0\}$. We have:
  \begin{propenum}
  \item If the spectral radius of the kernel $\kk$ is positive, then $I$ is
     non-empty.
  \item If the kernel $\kk$ is quasi-irreducible, then $\Omega_0=\{\kk\equiv
       0\}$ and $I$ is a singleton.
  \item\label{item:1-red} The kernel $\kk$ is monatomic if and only if $I$ is a singleton,
    say $I=\{\mathrm{a}\}$. Then the set $\oa$ is the atom of $\kk$.
  \item \label{rem:Aci-nv} If $A$ invariant implies $A^c$ invariant, then we have
    $\Omega_0=\{\kk\equiv 0\}$ and $\kk=\sum_{i\in I} \kk_i$ ($\kk$ reduced to $\Omega_0$
    is zero and intuitively $\kk$ is block diagonal).
  \item\label{item:Frob} The cardinal of the set of indices $i\in I$ such that
    $R_0[\kk_i]=R_0[\kk]$ is exactly equal to the multiplicity of $R_0[\kk]$ for $T_\kk$,
    that is $\mult(R_0[\kk], \kk)$.
  \item An eigenvalue $\lambda$ of $T_\kk$ is \emph{distinguished} if its distinguished
    multiplicity $\sharp\{ i\in I\, \colon\, R_0[\kk_i]=\lambda\}$ is positive. Notice
    that $R_0[\kk]$ is distinguished with its distinguished multiplicity equal to its
    multiplicity. Indeed if $R_0[\kk]$ is an eigenvalue of $\kk_i$, then it is its
    spectral radius and thus has multiplicity one as $\kk_i$ is quasi-irreducible. We also
    deduce that $\mult(R_0[\kk], \kk_i)\in \{0, 1\}$ for all $i\in I$.
  \end{propenum}
\end{remark}

For $i\in I$ and $\eta\in \Delta$, we set $\eta_i=\eta \ind{\Omega_i}$ and recall that
$\kk_i=\ind{\Omega_i} \kk \ind{\Omega_i}$. We now give the decomposition of $R_e[\kk]$
according to the quasi-irreducible components $(\kk_i, i\in I)$ of $\kk$. 

\begin{lemma}\label{lem:Rei}
  Let $\kk$ be a finite double norm kernel on $\Omega$ such that $R_0=R_0[\kk]>0$. We have
  for $\eta\in \Delta$:
  \begin{equation}\label{eq:R=maxRi}
    R_e[\kk](\eta)
    =\max_{i\in I} R_e[\kk_i](\eta_i)
    =\max_{i\in I} R_e[\kk](\eta \ind{\Omega_i}).
  \end{equation}
\end{lemma}

\begin{proof}
  For $A\in \cf$, recall  $\mult(\lambda, \kk, A)$ denotes the multiplicity (possibly
  equal to 0) of the eigenvalue $\lambda\in \C^*$ for the integral operator $T_{\kk
  \ind{A}}$ associated to the kernel $\kk\ind{A}$. Let $A, B\in \cf$ be such that $A\cap
  B=\emptyset$ a.s.\ and $\kk(B,A)=0$. Let $\eta \in \Delta$. Clearly we have $(\kk \eta)
  (B, A)=0$, and thus Lemma~\ref{lem:ext-61} gives that for all $\eta\in \Delta$:
  \[
    \mult(\lambda, \kk \eta, A\cup B)= \mult(\lambda, \kk \eta, A) + \mult(\lambda,
    \kk \eta, B). 
  \]
  Then, an immediate adaptation of the proof of \cite[Theorem~7]{schwartz61} gives that
  for all $\lambda\in \C^*$:
  \begin{equation}\label{eq:mult-vpFrob2}
    \mult(\lambda, \kk \eta, \Omega)=\sum_{i\in I} \mult(\lambda, \kk \eta, \Omega_i).
  \end{equation} 
  By definition of $\mult(\lambda, \cdot, \cdot)$, we get $
  R_e[\kk](\eta)=\max \{|\lambda|\, \colon\, \mult(\lambda, \kk \eta,
  \Omega)>0\}$ and $
  R_e[\kk\ind{\Omega_i}](\eta)= \max \{|\lambda|\, \colon\, \mult(\lambda,
  \kk \eta, \Omega_i)>0\}$.
  This gives that:
  \[
    R_e[\kk](\eta)= \max_{i\in I} R_e[\kk\ind{\Omega_i}](\eta ).
  \]
  To conclude, notice that $R_e[\kk](\eta \ind{\Omega_i} )=R_e[\kk\ind{\Omega_i}](\eta
  )=R_e[\ind{\Omega_i} \kk\ind{\Omega_i}](\eta )=R_e[\kk_i](\eta_i )$, where we used
  Lemma~\ref{lem:hk/h}~\ref{lem:hk=kh} for the second equality. 
\end{proof}

From Lemma~\ref{lem:Rei}, we deduce the following result.

\begin{lemma}\label{lem:conc-mono}
  Let $\kk$ be a finite double norm kernel on $\Omega$ such that $R_0=R_0[\kk]>0$. If the
  function $R_e[\kk]$ is concave on $\Delta$, then the kernel $\kk$ is monatomic.
\end{lemma}

\begin{proof}
  Since $R_0[\kk]$ is positive, we deduce that $\kk$ is not quasi-nilpotent. Suppose that
  $\kk$ is not monatomic. This means that the cardinal of the at most countable set $I$ in
  the decomposition~\eqref{eq:R=maxRi} is at least $2$. So let $\kk_{1}$ and $\kk_{2}$ be
  two quasi-irreducible components of $\kk$, where we assume that $\{1,2 \} \subset I$.
  Let $\Omega_{1}$ and $\Omega_{2}$ denote their respective atoms. Without loss of
  generality, we can suppose that $R_0[\kk_{2}] \geq R_0[\kk_{1}]>0$. Consider the
  strategies $\eta_1=\ind{\Omega_{1}}$ and $\eta_2=R_0[\kk_{1}] \, R_0[\kk_{2}]^{-1}\,
  \ind{\Omega_{2}}$ (which both belong to $\Delta$). For $\theta\in [0, 1]$, we deduce
  from \eqref{eq:R=maxRi} and the homogeneity of the spectral radius that $R_e[\kk](\theta
  \eta_1 + (1-\theta) \eta_2)= R_e[\kk_1] \max(\theta, 1-\theta)$. Since $\theta\mapsto
  \max(\theta, 1-\theta)$ is not concave, we deduce that $R_e[\kk]$ is not concave on
  $\Delta$.
\end{proof}

Set $\tilde \kk=\sum_{i\in I} \kk_i$. As a consequence of~\eqref{eq:mult-vpFrob2}, we have
that:
\begin{equation}\label{eq:kk=tilde-kk}
  \spec[\kk]=\spec[\tilde \kk] \quad\text{and}\quad R_e[\kk]=R_e[\tilde \kk].
\end{equation}
In view of Section~\ref{sec:equivalent}, \eqref{eq:kk=tilde-kk} gives an other
transformation of the kernel $\kk$ which leaves the function $\spec[\kk]$ unchanged. We
represent in Figure~\ref{fig:atomic1} an example of a kernel $\kk$ with its atomic
decomposition using $\preccurlyeq$ as a partial order on $\Omega$ and in
Figure~\ref{fig:atomic2} the corresponding kernel $\tilde \kk$.

\begin{figure}
  \begin{subfigure}[T]{.5\textwidth}
    \centering
    \includegraphics[page=1]{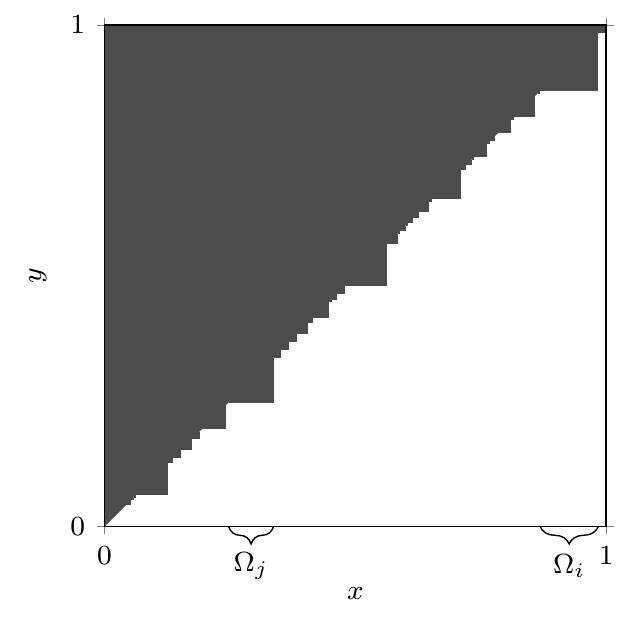}
    \caption{A representation of the kernel $\kk$.}
    \label{fig:atomic1}
  \end{subfigure}%
  \begin{subfigure}[T]{.5\textwidth}
    \centering
    \includegraphics[page=2]{reducible}
    \caption{A representation of the kernel $\tilde \kk=\sum_{i\in I}
      \kk_i$. We have $\spec[\kk]=\spec[\tilde \kk]$ and thus
    $R_e[\kk]=R_e[\tilde \kk]$.}
    \label{fig:atomic2}
  \end{subfigure}
  \caption{Example of a kernel $\kk$ with the white zone included in
    $\{\kk=0\}$ and the kernel $\tilde\kk =\sum_{i\in I} \kk_i$, with
    $\kk_i=\ind{\Omega_i}\kk\ind{\Omega_i}$ and $\kk(\Omega_i,
    \Omega_j)=0$ for $j\prec i $.}
  \label{fig:atomic}
\end{figure}

We set $R_0=R_0[\kk]$. For $i\in I$, we consider the loss $R_e[\kk_i]$ and the
corresponding optimal loss function $R^\star_i$ defined on $[0, \cmax]$ and optimal cost
function $C^\star_i$. For convenience the function $C^\star_i$ which is defined on $[0,
R_0[\kk_i]]$ is extended to $[0, R_0]$ by setting $C^\star_i=0$ on $(R_0[\kk_i], R_0]$. 
Notice also that $\{\kk_i\equiv 0\}=\Omega_i^c$. Recall that $\cmax=C(\un)$. We now state
the main result of this section, which in particular gives a description of the
anti-Pareto frontier. 

\begin{corollary}\label{cor:loss-pb}
  Suppose that the cost function $C$ is continuous decreasing with $C(\un)=0$  and
  consider the loss function $R_e[\kk]$, with $\kk$ a finite double norm kernel on
  $\Omega$ such that $R_0=R_0[\kk]>0$. We  have:
  \[
    R_e^\star= \max_{i\in I}\, R^\star_i \quad\text{(on $[0, \cmax]$), and}\quad
    \Csup= \max_{i\in I}\, C^\star_i \quad\text{(on $[0, R_0]$)};
  \]
  the maximal cost of totally inefficient strategies is given by:
  \[
    \cmar=\Csup(R_0)=\max_{ i\in I} \{C(\ind{\Omega_i})\, \colon\,
    R_0[\kk_i]=R_0[\kk]\};
  \]
  and the anti-Pareto frontier is given by: 
  \begin{equation}\label{eq:rep-AF}
    \AF = \left\{ \left( \max_{i\in I} C^\star_i (\ell), \ell\right) \,
    \colon \, \ell \in [0, R_0] \right\}.
  \end{equation}
  Furthermore, we have for $\ell\in [0, R_0]$:
  \[
    \Cinf(\ell)=C\left(\eta_\star\right)
    \quad\text{with}\quad \eta_\star=\ind{\Omega_0}+ \sum_{i\in I} \eta_{i,\star},
  \]
  where $\eta_\star$ is Pareto optimal with $R_e[\kk](\eta_\star)=\ell$, and, for $i\in
  I$, the strategy $\eta_{i, \star}=\eta_\star\, \ind{\Omega_i}$ restricted to $\Omega_i$
  is  Pareto optimal for the kernel $\kk_i$ restricted to $\Omega_i$, with
  $R_e[\kk_i](\eta_{i, \star})=\min(\ell, R_0[\kk_i])$. We also have an upper bound for
  the minimal cost which ensures that no infection occurs at all:
  \[
    \cmir=\Cinf(0)\leq C(\ind{\Omega_0}).
  \]
\end{corollary}

\begin{rem}\label{rem:C*_(0)=}
  We easily deduce from the previous corollary that $\Cinf(0)$ is in fact equal to the
  cost of $\ind{\Omega_0\cup A}$ where $A=\cup_{i\in I} A_i$ and, for all $i\in I$, 
  $A_i\subset \Omega_i$ is a $C$-maximal independent set associated to the kernel $\kk_i$,
  see Remark~\ref{rem:alpha=C(A)}.
\end{rem}

\begin{remark}\label{rem:L-discont}
  If $\kk$ is not monatomic, then Assumption 7 in \cite{ddz-theo} (that is any local
  maximum of the loss function is also a global maximum) may or may not be satisfied for
  the loss function $\loss=R_e[\kk]$, see the case of the two population model in
  \cite{ddz-2pop}. In the former case the function $C ^\star$ is continuous and the
  anti-Pareto frontier is connected, whereas in the latter case the function $C ^\star$
  may have jumps and then the anti-Pareto frontier has more than one connected component. 
\end{remark}

\begin{proof}
  Equation \eqref{eq:R=maxRi} and the definition of $R_e^\star$ readily implies that $
  R_e^\star= \max_{i\in I}\, R^\star_i$.

  We set $R_0=R_0[\kk]$ and recall that $R_e[\kk_i](\un)=R_0[\kk_i]$. Let $\ell \in (0,
  R_0]$. Notice that \eqref{eq:mult-vpFrob} implies that there is a finite number of
  indices $i\in I$ such that $R_0[\kk_i]\geq \ell$. This and \eqref{eq:R=maxRi} readily
  implies that $ \Csup(\ell)= \max_{i\in I}\, C^\star_i(\ell)$ for $\ell>0$. Use that
  $\Csup(0)=C ^\star_ i(0)=\cmax$ to deduce that the equality $ C^\star= \max_{i\in I}\,
  C^\star_i$ holds on $[0, R_0]$. The formula for $\cmar=\Csup(R_0)$ is a consequence of
  \eqref{eq:R=maxRi}, Lemma \cite[Lemma~5.14]{ddz-theo} and
  Remark~\ref{rem:Frob}~\ref{item:Frob}. The formula~\eqref{eq:rep-AF} for
  $\mathcal{F}^\mathrm{Anti} $ is then a consequence of \eqref{eq:FL=L*}.

  Eventually, if $\eta_\star$ is Pareto optimal with $R_e[\kk](\eta_\star)=\ell$, we
  deduce from~\eqref{eq:R=maxRi} that $R_e[\kk](\eta_\star \ind{\Omega_0^c})$ is also
  equal to $\ell$, and since $C$ is decreasing, this implies that $\eta_\star \geq
  \ind{\Omega_0}$ and thus $\eta_\star= \ind{\Omega_0} + \sum_{i\in I} \eta_{i,\star}$
  with $\eta_{i, \star}=\eta_\star\, \ind{\Omega_i}$. Now if $\eta_{i, \star}$ were not
  Pareto optimal for the kernel $\kk_i$ restricted to $\Omega_i$ or if
  $R_e[\kk_i](\eta_{i, \star})<\min(\ell, R_0[\kk_i])$, we could increase $\eta_\star$ on
  $\Omega_i$ without changing the value of $R_e[\kk]$, and thus $\eta_\star$ would not be
  Pareto optimal. Thus, we get that $\eta_{i, \star}$ is Pareto optimal for the kernel
  $\kk_i$ restricted to $\Omega_i$, that is, $\eta_{i, \star}+ \ind{\Omega_0}$ is Pareto
  optimal for the kernel $\kk_i$, and that $R_e[\kk_i](\eta_{i, \star})=\min(\ell,
  R_0[\kk_i])$. From the inequality $\eta_\star \geq \ind{\Omega_0}$, we deduce that $
  \cmir=\Cinf(0)\leq C(\ind{\Omega_0})$. 
\end{proof}

\printbibliography

\end{document}